\documentclass[12pt, a4paper]{article}

\usepackage{amsmath}
\usepackage{amsfonts}
\usepackage{amssymb}
\usepackage[ansinew]{inputenc}
\usepackage[pdftex]{graphicx}
\usepackage[english]{babel}

\usepackage{alltt}
\usepackage{color}
\usepackage{amsthm}

\definecolor{string}{rgb}{0.7,0.0,0.0}
\definecolor{comment}{rgb}{0.13,0.54,0.13}
\definecolor{keyword}{rgb}{0.0,0.0,1.0}

\fussy

\theoremstyle{plain}
\newtheorem{theorem}{Theorem}[section]
\newtheorem{lemma}[theorem]{Lemma}

\newtheorem{remark}[theorem]{Remark}

\title{An interface formulation for the Poisson equation\\ in the presence of a semiconducting single-layer material}
\author{C. Jourdana$^{1}$ and P. Pietra$^2$}
\date{}

\begin{document}

\maketitle

\vspace{-0.8cm}\begin{small}
    \begin{center}
$^1$ Univ. Grenoble Alpes, CNRS, Grenoble INP$^{\dag}$, LJK, 38000 Grenoble, France \\
$^{\dag}$ Institute of Engineering Univ. Grenoble Alpes\\
$^2$ Istituto di Matematica Applicata e Tecnologie Informatiche ``E. Magenes'' - CNR\\
Via Ferrata 1, 27100 Pavia, Italy\\
\it{clement.jourdana@univ-grenoble-alpes.fr ; pietra@imati.cnr.it}
    \end{center}
\end{small}

\begin{abstract}
In this paper, we consider a semiconducting device with an active zone made of a single-layer material. The associated Poisson equation for the electrostatic potential (to be solved in order to perform self-consistent computations) is characterized by a surface particle density and an out-of-plane dielectric permittivity in the region surrounding the single-layer. To avoid mesh refinements in such a region, we propose an interface problem based on the natural domain decomposition suggested by the physical device. Two different interface continuity conditions are discussed. Then, we write the corresponding variational formulations adapting the so called three-fields formulation for domain decomposition and we approximate them using a proper finite element method. Finally, numerical experiments are performed to illustrate some specific features of this interface approach.
\end{abstract}





\section{Introduction}

Two-dimensional (2D) materials such as the most well-known graphene are crystal structures made of a single layer of atoms. With the recent progress to isolate, stack and characterize them, they are promising for a wide range of applications (see e.g. the reviews \cite{XLSC13,B_al13}). In particular they become an option to design post-silicon nanoelectronic devices. Field-Effect Transistors (FETs) based on graphene (GFETs) or, more generally, on semiconducting 2D materials (2D-FETs) give the possibility to have a channel thickness on the atomic scale which ideally should reduced short-channel effects while maintaining high carrier mobility. However, the performance of various 2D-FETs is still difficult to predict and accurate numerical simulations can take part in a better understanding.

A first focus is on transport properties in such a device. For instance, graphene is characterized by a zero bandgap and chiral massless carriers. It leads to unusual transport properties such as integer quantum Hall effect or Klein tunneling \cite{CGPNG}. Different transport models have been recently derived or investigated ranging from the two dimensional  Dirac equation \cite{CGPNG,FW14} to sophisticated drift-diffusion and hydrodynamical systems such that e.g. \cite{ZB11,ZJ13,LR19,NR21}. Another focus is on the Poisson equation for the electrostatic potential that has to be solved to perform self-consistent computations. In particular, the dielectric response of 2D layered structures has to be properly taken into account. It is this aspect that we tackle in this work proposing to model the single-layer as an interface and leading to a Poisson problem that can be solved numerically in an efficient way.

More precisely, we consider a device with an active zone made of a single-layer material sandwiched between two thick insulator regions (oxide). The associated Poisson equation is characterized by a surface particle density and an out-of-plane dielectric permittivity exhibited in a region of effective dielectric thickness surrounding the single-layer material, as discussed in \cite{FVF16}. Both these characteristics require an extremely fine mesh around the 2D material in order to provide an accurate approximate solution of this equation. To avoid it, we propose, averaging the potential across the dielectric effective region, an interface problem based on the natural domain decomposition suggested by the physical device. It is made of two Laplace equations in the oxide subdomains coupled with an effective Poisson equation on the interface with an extra source term that represents the contribution of the surrounding environment to the channel material. This approach is inspired by \cite{AJRS02} where it is used to model fractures in porous media. It is worth mentioning that, contrary to compact models of GFETs (e.g. \cite{JM11,UKV18}), it leads to a full multidimensional Poisson problem.

For the treatment of the Poisson equation in self-consistent models for graphene based devices, we recall also \cite{NR21}, where authors assume that the carrier charge is uniformly distributed in the volume between the two oxide regions and \cite{EHM14}, where authors prove existence and uniqueness results for a Dirac-Poisson problem and consider the self-consistent potential as the trace in the plane of the graphene of the 3D Poisson potential and thus as the solution of a fractional Laplacian equation. Finally, we mention \cite{NBCD10} where the Poisson equation is written in an integral form and the method of moments is used.
 
In order to match the interface potential to the oxide potentials, we first consider a simple continuity condition, obtaining an interface model that indeed takes into account the effective dielectric thickness, but it does  not retain the information of the out-of-plane permittivity when a channel dielectric diagonal tensor is considered. That is why we also introduce a Robin type continuity condition, following a work on fractured porous media \cite{MJR05} again.

A discrete fracture-matrix model for flow in porous media is considered in \cite{KMJR19}, where the exchange between the fracture and the matrix is imposed using a Lagrange multiplier, in the spirit of a fictitious domain approach. Here, to analyze and discretize our interface model, we write the corresponding variational formulations adapting the so called three-fields formulation for domain decomposition in the form introduced and analyzed in \cite{B03} (see also \cite{BM01,BK00}). It is a non conforming formulation of
non-overlapping domain decomposition that introduces the space of traces of functions in $H^1(\Omega)$ on the interface. The weak continuity between the 
2D subdomains and the interface is then imposed by means of Lagrange multipliers. This variational formulation enters in the framework of saddle point problems \cite{BF} which gives existence and uniqueness results as well as error estimates when the problem is approximated using a proper finite element method. Interestingly, the interface discretization does not need to match with the one of the subdomains and we take advantage of this flexibility in the numerical experiments we are performing to illustrate the approach.

The paper is organized as follows. The interface model with the two continuity conditions are introduced in Section 2. The variational formulation of the problem with the simple continuity condition adapted from the so called three-fields formulation is presented and analyzed in Section 3 and then discretized in Section 4. Section 5 is dedicated of the Robin type continuity condition that can be used to tackle an anisotropic permittivity. Finally, some numerical experiments are performed in Section 6.


\section{Interface model presentation}

As we said, we consider a device with an active zone made of a single-layer material sandwiched between two oxide regions. We assume the single-layer is large enough to be just considered as a one dimensional (1D) line along the direction $x$, the transport along the other direction being free and boundary effects being neglected. We denote by $y$ the direction perpendicular to the single-layer plane made of oxide/single layer/oxide slices. It gives a 2D domain $\Omega=]0,L[\times ]-\frac{l}{2},\frac{l}{2}[$ where $L$ is the longitudinal device length and $l$ the transversal one.

The electrostatic potential $u$ created by such a device is solution to the 2D Poisson equation
\begin{equation}\label{eq_poisson}
- \nabla \cdot \Big(\epsilon(x,y) \nabla u(x,y) \Big)= \rho(x) \delta(y), \hspace{0.5cm} \hbox{ in } \Omega,
\end{equation}
where $\rho$ is the surface particle density, $\delta$ the Dirac distribution imposing that the particles are confined to the single-layer plane and $\epsilon$ the dielectric permittivity. This equation is completed by boundary conditions. We assume that the boundary $\partial \Omega$ splits into two parts: the Ohmic contacts $\Gamma_D$ and the insulating parts $\Gamma_N$, with $\partial \Omega=\Gamma_D \cup \Gamma_N$ and $\Gamma_D \cap \Gamma_N=\emptyset$. The potential is prescribed on $\Gamma_D$ while there are no-flux boundary conditions on $\Gamma_N$:
\begin{equation}\label{cond_dirichlet}
u=u_D, \hspace{0.5cm} \hbox{ on } \Gamma_D,
\end{equation}
\begin{equation}\label{cond_neumann}
\nabla u \cdot \nu=0, \hspace{0.5cm} \hbox{ on } \Gamma_N,
\end{equation}
where $\nu$ is the outward unit normal on $\Gamma_N$ and $u_D$ represents Source, Drain and Gate potentials. Since Source and Drain contacts touch the single-layer material, we assume that $\Gamma_D$ contains the single-layer boundary points.

Due to single-layer/oxide interactions, the permittivity in the oxide is affected in a region surrounding the single-layer material. The choice of the permittivity $\epsilon(x,y)$ is a delicate modeling issue. Here, we introduce an effective dielectric thickness $d$ and we  assume one dielectric constant for the channel and another one for the oxide:
\begin{equation*}
\epsilon(x,y)=\begin{cases} \epsilon_{ch} \hbox{ for } |y| < \frac{d}{2} \\ \epsilon_{ox} \hbox{ otherwise} \end{cases}.
\end{equation*}
Such an approach has been used in \cite{OYG07,FI07}, e.g, and it is often referred to as ``box assumption''. The somehow arbitrariness in the choice of the discontinuity lines in \cite{OYG07,FI07} is mitigated by using the results in \cite{FVF16}, where studies of the atomic-scale Poisson equation provide values for the dielectric thickness, validating somehow the ``box assumption''.

Our objective is not to deal directly with the computationally demanding transmission problem that consists in imposing, along $\gamma_{\pm}=\{ (x,\pm \frac{d}{2}), \ x \in ]0,L[\}$, continuity of the potential and of the transversal electric displacement, as summarized in the following equations:
\begin{equation}\label{eq1_strip}
- \nabla \cdot (\epsilon_{ox} \nabla u_{\pm} )=0, \hspace{0.5cm} \hbox{ in } ]0,L[ \times \big( \pm ]{\textstyle \frac{d}{2}, \frac{l}{2}}[ \big),
\end{equation}
\begin{equation}\label{eq2_strip}
- \nabla \cdot (\epsilon_{ch} \nabla u_{ch} )= \rho(x) \delta(y), \hspace{0.5cm} \hbox{ in } ]0,L[\times ]{\textstyle-\frac{d}{2},\frac{d}{2}}[,
\end{equation}
\begin{equation}\label{cond_transm_strip}
u_{\pm} =u_{ch}, \quad \epsilon_{ox}  \partial_y u_{\pm} = \epsilon_{ch} \partial_y u_{ch} \quad \hbox{on } \gamma^{\pm}
\end{equation}
where $\delta$ is the Dirac delta function.

Instead, inspired by \cite{AJRS02} to model fractures in porous media, we propose to consider an interface problem obtained by averaging the potential across the dielectric effective region and considering $d$ small enough  to assume a matching between $\gamma_{+}$ and $\gamma_{-}$. Introducing
$$u_{\gamma}(x) =\frac{1}{d} \int_{-\frac{d}{2}}^{\frac{d}{2}} u_{ch}(x,y) \ dy,$$   
performing integration  in the transversal direction of equation  \eqref{eq2_strip} and using the flux continuity in \eqref{cond_transm_strip}, we obtain the 1D effective equation
$$-d (\epsilon_{ch} u_{\gamma}')' = \rho - \epsilon_{ox} \Big(\nabla u_1(x,0) \cdot n_1 + \nabla u_2(x,0) \cdot n_2\Big), \hspace{0.5cm} \hbox{ in } \gamma, $$
where $\gamma=]0,L[$ represents the single-layer line, $u_i$, $i=1,2$, are the  potentials associated to each oxide subdomain $\Omega_i$ ($\Omega_1=]0,L[ \times]0,\frac{l}{2}[$ and $\Omega_2=]0,L[\times]-\frac{l}{2},0[$) and $n_i$ are the two outward unit normals on $\partial \Omega_i \cap \gamma$. One should notice in this 1D equation the presence of the dielectric thickness $d$. Indeed, $-d \epsilon_{ch} u_{\gamma}'$ represents the electric displacement through the cross section of the dielectric effective region. Also, we emphasize that the extra source term appearing in the right-hand side (in addition to the 1D charge density $\rho$) represents the contribution to the interface of the transversal electric displacement from the surrounding environment. 

Consequently, the interface model that we analyze and discretize in the next sections consists in two Laplace equations in the oxide subdomains
\begin{equation}\label{eq_oxide}
- \nabla \cdot  (\epsilon_{ox}  \nabla u_i) = 0, \quad \hbox{in}~ \Omega_i, ~~i=1,2,
\end{equation} 
and the effective Poisson equation in the single-layer line
\begin{equation}\label{eq_interf}
- d (\epsilon_{ch} u_{\gamma}' )'  =  \rho -  \epsilon_{ox}(\nabla u_1\cdot n_1 +  \nabla  u_2 \cdot n_2), \quad \hbox{in}~ \gamma,
\end{equation}
the potentials associated to the three domains being connected by the continuity conditions
\begin{equation}\label{eq_continuity}
u_i= u_{\gamma} , \quad \hbox{on} ~\gamma, ~~i=1,2.
\end{equation}
This system is completed by the following mixed boundary conditions for the oxide potentials
\begin{equation}\label{cond_dirichlet2}
u_{i}= u_D, \quad \hbox{ on } \Gamma^i_D=\partial \Omega_i \cap \Gamma_D, ~~i=1,2,
\end{equation}
\begin{equation}\label{cond_neumann2}
\nabla u_{i} \cdot \nu_i =0, \quad \hbox{ on } \Gamma^i_N=\partial \Omega_i \cap \Gamma_N, ~~i=1,2,
\end{equation}
$\nu_i$ being the outward unit normal on $\partial \Omega \cap \partial \Omega_i$,
and by the following Dirichlet boundary condition for the interface potential
\begin{equation}\label{cond_dirichlet_interf}
u_{\gamma}(0)=u_D(0,0), \quad   u_{\gamma}(L)=u_D(L,0).
\end{equation}

In a more physically relevant setting the channel dielectric permittivity is given by a diagonal tensor 
\begin{equation}\label{perm_tensor}
\epsilon_{ch}= \begin{pmatrix} \epsilon_{//} & 0 \\ 0 & \epsilon_{\perp} \end{pmatrix} 
\end{equation}
rather than by a dielectric constant, introducing an in-plane permittivity  $\epsilon_{//}$ and an out-of-plane permittivity $\epsilon_{\perp}$. In that case, in the effective equation \eqref{eq_interf} only $\epsilon_{//}$ appears. To retain the information about $\epsilon_{\perp}$, we replace the continuity conditions \eqref{eq_continuity} by a Robin type condition as done in \cite{MJR05} to model fractures in porous media. Formally, we  say that
$$u_{\gamma}(x) \approx u_{ch}(x,\pm \frac{d}{2}) \mp \frac{d}{2}\partial_y u_{ch}(x,\pm \frac{d}{2})$$
and we use this approximation into the continuity of the transversal electric displacement along $\gamma_{\pm}$ \eqref{cond_transm_strip}. It gives
$$\epsilon_{ox}  \partial_y u_{\pm} = \epsilon_{\perp} \partial_y u_{ch} \approx \pm \epsilon_{\perp}  \frac{u_{\pm} - u_{\gamma}}{d/2}.$$
Assuming a matching between $\gamma_{+}$ and $\gamma_{-}$, we obtain the Robin type condition
\begin{equation}\label{eq_continuity2}
(u_i-u_{\gamma})+\alpha \ \epsilon_{ox} \nabla u_i \cdot n_i = 0, \quad \hbox{on} ~\gamma, ~~i=1,2,
\end{equation}
with $\alpha= \frac{d}{2\epsilon_{\perp}}$.
As we will see in Section \ref{section_robin}, this Robin condition at interface changes only slightly the mathematical analysis. Moreover, a numerical comparison of the two continuity conditions \eqref{eq_continuity} and \eqref{eq_continuity2}  will be performed in Section \ref{section_num}.


\section{Variational formulation}\label{int_form}

Let us first introduce some notation needed in the rest of the paper. For any domain $\widehat{\Omega}$ and $m\geq 0$, we denote by $\| \cdot \|_{m,\widehat{\Omega}}$ the $H^m(\widehat{\Omega})$ norm. For a convex Lipschitz $\Omega  \subset \mathbb{R}^2$, we denote by $\Gamma_0$ and $\Gamma_1$ two subsets of the boundary, with $\partial \Omega = \Gamma_0 \cup \Gamma_1$ and $\Gamma_0 \cap \Gamma_1 = \emptyset$. We shall employ the notation $H^1_{0,\Gamma_1}(\Omega)=\{ v \in H^1(\Omega), v=0 \hbox{ on } \Gamma_1\}$ and define
$H^{1/2}_{00} (\Gamma_0)$  as the trace space of $H^1_{0,\Gamma_1}(\Omega)$ equipped with the norm
\begin{equation}\label{def_norm_1/2_Sigma}
\| \sigma \|_{1/2,\Gamma_0}= \inf_{v \in H^1_{0,\Gamma_1}, \ v|_{\Gamma_0}=\sigma } \| v \|_{1,\Omega},
\end{equation}
and we shall denote by $( . , .)_{1/2,\Gamma_0}$ the corresponding inner product. 

Finally, duality between $H^{1/2}_{00} (\Gamma_0)$ and its dual space $\Big( H^{1/2}_{00} (\Gamma_0) \Big)'$ is written $<.,.>_{\Gamma_0} $ and we shall use as norm in the dual space the equivalent norm  $\|.\|_{-1/2,\Gamma_0}$ defined as:
\begin{equation}\label{def_norm_m1/2_Sigma}
\| \mu \|_{-1/2, \Gamma_0}= \sup_{v \in H^{1}_{0,\Gamma_1} (\Omega) } \frac{<\mu,v>_{\Gamma_0}}{\| v \|_{1,\Omega}}.
\end{equation}
Also, we denote by $C>0$ a generic constant with values that may change from line to line.

\begin{remark} \label{remark_equivalency}
Notice that the norm \eqref{def_norm_m1/2_Sigma} is equivalent to the dual norm defined by
$$\sup_{\sigma \in H^{1/2}_{00} (\Gamma_0) } \frac{<\mu,\sigma>_{\Gamma_0}}{\| \sigma \|_{1/2,\Gamma_0}}.$$
Indeed, on one hand, given $v \in H^1_{0,\Gamma_1}$, its trace  on $\Gamma_0$ (still denoted $v$)
is in  $ H^{1/2}_{00} (\Gamma_0)$ and verifies
$$\| v \|_{1/2,\Gamma_0} \leq C \| v \|_{1,\Omega}.$$
Therefore, for all $v\in H^{1}_{0,\Gamma_1} (\Omega)$,
$$
 \frac{<\mu,v>_{\Gamma_0}}{\| v \|_{1,\Omega}} \le C \sup_{\sigma \in H^{1/2}_{00} (\Gamma_0) } \frac{<\mu,\sigma>_{\Gamma_0}}{\| \sigma \|_{1/2,\Gamma_0}}.
$$
On the other hand, given $\sigma \in H^{1/2}_{00} (\Gamma_0)$, we can construct a lifting function in $H^1_{0,\Gamma_1}$, denoted $v_{\sigma}$, such that $v_{\sigma}|_{\Gamma_0} = \sigma$ and $-\Delta v_{\sigma} + v_{\sigma} =0$ in $\Omega$. Then, we have
$$<\mu,\sigma>_{\Gamma_0}= <\mu,v_{\sigma}|_{\Gamma_0} >_{\Gamma_0}$$ 
and
$$ \| \sigma \|_{1/2,\Gamma_0} = \| v_{\sigma} \|_{1,\Omega}.$$
Therefore, for all $\sigma \in H^{1/2}_{00} (\Gamma_0)$,
$$
\frac{<\mu,\sigma>_{\Gamma_0}}{\| \sigma \|_{1/2,\Gamma_0}} \le \sup_{v \in H^{1}_{0,\Gamma_1} (\Omega) } \frac{<\mu,v>_{\Gamma_0}}{\| v \|_{1,\Omega}}.
$$
\end{remark}

\noindent For simplicity of the presentation, we consider the problem \eqref{eq_oxide}-\eqref{cond_dirichlet_interf} with homogeneous Dirichlet conditions on $\Gamma_D$. It writes \\
\noindent {\it Find $(u_1,u_2,u_{\gamma})$ s.t.}
\begin{equation}\label{interf_pb_value_pb}
\left \{
\begin{array}{r c l @{\hspace{1cm}} l}
- \nabla \cdot  (\epsilon_{ox}  \nabla u_i) &=& 0  &\hbox{in}~ \Omega_i, ~~i=1,2,\\
- d (\epsilon_{ch} u_{\gamma}' )'  &=&  \rho -  \epsilon_{ox}(\nabla u_1\cdot n_1 +  \nabla  u_2 \cdot n_2)  & \hbox{on}~ \gamma,\\
u_i&=& u_{\gamma} & \hbox{on} ~\gamma,  \\
u_{i}&=&0 &\hbox{on } \Gamma^i_D, \\
\epsilon_{ox}\nabla u_{i} \cdot n_i &=&0 &\hbox{on } \Gamma^i_N,\\
u_{\gamma}(0)=u_{\gamma}(L)&=&0.
\end{array}
\right. 
\end{equation}
The functional setting we choose in order to write a variational formulation of the interface problem \eqref{interf_pb_value_pb} is the following.  We define the spaces:\\
- ${\bf V}=V^1 \times V^2 \times V^{\gamma}$, with 
$V^i=H^1_{0,\Gamma_D^i} (\Omega_i)$, $i=1,2$  and $ V^{\gamma}=H^1_0(\gamma),$
equipped with the norm
$$
||{\bf u}||_{{\bf V}} = \Big( \sum_{i=1}^2 ||u_i||^2_{1,\Omega_i}+||u_{\gamma}||^2_{1,\gamma}\Big)^{1/2},
$$
- $\boldsymbol \Lambda=\Lambda^1 \times \Lambda^2$, with $\Lambda^i =\big(H^{1/2}_{00}(\gamma)\big)'$ equipped with the norm
$$
||\boldsymbol \lambda ||_{\boldsymbol\Lambda} = \Big( \sum_{i=1}^2 ||\lambda_i ||^2_{-1/2, \gamma} \Big)^{1/2},
$$
where
$ \displaystyle \| \lambda_i \|_{-1/2, \gamma}= \sup_{v \in H^1_{0,\partial \Omega_i \backslash \gamma} (\Omega_i)} \frac{<\lambda_i,v>_{\gamma}}{\| v \|_{1,\Omega_i}}$.

\bigskip\noindent For ${\bf u}=(u_1,u_2,u_{\gamma}) \in {\bf V}$, ${\bf v}=(v_1,v_2,v_{\gamma}) \in {\bf V}$ and $\boldsymbol \mu=(\mu_1,\mu_2) \in \boldsymbol \Lambda$, we define the bilinear forms
$$
{\bf a}({\bf u},{\bf v})=\sum_{i=1}^2 \int_{\Omega_i} \epsilon_{ox}  \nabla u_i \cdot \nabla v_i  \, dx dy+d \int_{\gamma} \epsilon_{ch}  u'_{\gamma} v'_{\gamma}\, dx, $$
$$
{\bf b}(\boldsymbol \mu,{\bf u})=\sum_{i=1}^2 < \mu_i , u_i|_{\gamma} - u_{\gamma} >_{\gamma} .\qquad 
$$
Notice that $u_i \in V^{i}$ implies $u_i|_{\gamma} \in H^{1/2}_{00} (\gamma)$ (see \cite{LM,F00}). Therefore, with $u_{\gamma} \in V^{\gamma}$, the duality pairing is meaningful. In the following, we will use $u_i$ instead of  $u_i|_{\gamma} $ in the duality pairing unless it might create some confusion.

\medskip\noindent
We consider the following variational problem:\\
{\bf Variational formulation:}\\
\noindent Find $({\bf u},\boldsymbol \lambda) \in {\bf V} \times \boldsymbol \Lambda$ s.t.
\begin{equation} \label{interf_VF}
\left \{
\begin{array}{r c l @{\hspace{1cm}} l}
{\bf a}({\bf u},{\bf v}) - {\bf b}(\boldsymbol \lambda,{\bf v}) &=& \int_{\gamma} \rho \, v_{\gamma}\, dx,   & \forall {\bf v} \in {\bf V},\\
{\bf b}(\boldsymbol \mu,{\bf u}) &=& 0, & \forall \boldsymbol \mu \in \boldsymbol\Lambda. 
\end{array}
\right. 
\end{equation}
In this formulation, the continuity $u_i= u_{\gamma}$ on $\gamma$ is imposed as a constraint through the Lagrange multipliers $\boldsymbol \lambda$. The first equation is associated to the two Laplace equations in the oxide subdomains as well as the effective Poisson equation on the interface. Indeed, a regular solution  (${\bf u}$, $\boldsymbol \lambda$) to \eqref{interf_VF} is linked to a solution to \eqref{interf_pb_value_pb} in the following sense. Taking $v_{\gamma}=0$ in the first equation of  \eqref{interf_VF} gives
\begin{equation*}
\int_{\Omega_i} \epsilon_{ox}  \nabla u_i \cdot \nabla v_i  \, dx dy - < \lambda_i , v_i >_{\gamma} = 0 \qquad \forall v_i \in V^i, \quad i=1,2.
\end{equation*}
Since $\overline{\Gamma}_N^i \cap \overline{\gamma}$ is empty, a Green formula gives for $v_i \in V^i$
$$\int_{\Omega_i} \epsilon_{ox}  \nabla u_i \cdot \nabla v_i  \, dx dy = -\int_{\Omega_i} \nabla \cdot (\epsilon_{ox}  \nabla u_i ) v_i  \, dx dy +  < \epsilon_{ox}  \nabla u_i \cdot n_i , v_i >_{\gamma} +  < \epsilon_{ox}  \nabla u_i \cdot n_i , v_i >_{\Gamma_N^i}.$$
Choosing first $v_i \in H^1_0(\Omega_i) \subset V^i$, we obtain $-\nabla \cdot  (\epsilon_{ox}  \nabla u_i) = 0,$ $a.e.$ in $\Omega_i$. Then, for $v_i \in H^1_{0,\gamma \cup \Gamma_D^i}(\Omega_i) \subset V^i$, we have $v_i|_{\Gamma_N^i} \in H^{1/2}_{00}(\Gamma_N^i)$ and consequently
\begin{equation*}
< \epsilon_{ox}\nabla u_i \cdot n_i ,v_i>_{\Gamma_N^i}=0 , \hspace{0.5cm} \hbox{for all } v_i \in H^{1/2}_{00}(\Gamma_N^i).
\end{equation*}
Next, for $v_i \in V^i$, we obtain
 \begin{equation}\label{conormal}
<\lambda_i,v_i>_{\gamma} = < \epsilon_{ox}\nabla u_i \cdot n_i ,v_i>_{\gamma} , \hspace{0.5cm} \hbox{for all } v_i \in H^{1/2}_{00}(\gamma).
\end{equation}
It links $\lambda_i$ to $\epsilon_{ox}\nabla u_i \cdot n_i $. Finally, taking $v_i=0$ for $i=1,2$ in the first equation of  \eqref{interf_VF} and using \eqref{conormal}, we obtain
\begin{equation*}
d \int_{\gamma} \epsilon_{gr}^l  u'_{\gamma} v'_{\gamma}\, dx + \sum_{i=1}^2 <\epsilon_{ox}\nabla u_i \cdot n_i   ,v_{\gamma} >_{\gamma} = \int_{\gamma} \rho v_{\gamma}\, dx
\end{equation*}
which is a  weak form for the second equation of \eqref{interf_pb_value_pb}. The second equation of \eqref{interf_VF} imposes the continuity $u_i= u_{\gamma}$  on $\gamma$ in a weak form.

\begin{remark}
Formulation \eqref{interf_VF} is an adaptation to the interface problem of the so called three-fields-formulation in the form introduced and analyzed in \cite{B03} (see also \cite{BM01,BK00}). We notice however, that, for the peculiarity of our setting that provides directly coercivity of the bilinear form ${\bf a}({\bf u},{\bf v})$  on the whole space ${\bf V}$, we don't really introduce three fields, but rather work with two spaces only: ${\bf V}$ (space for the potentials on $\Omega_i$'s and on $\gamma$) and $\boldsymbol \Lambda$ (Lagrange multipliers for the Dirichlet BC's on $\gamma$, to be interpreted as conormal derivative of $u_i$ as seen in \eqref{conormal}). 
\end{remark}

\medskip\noindent Existence and uniqueness results follow from  the theory for saddle point problems \cite{BF} as stated by Theorem \ref{theorem1}, thanks to the properties of the bilinear forms ${\bf a}({\bf u},{\bf v})$ and ${\bf b}(\boldsymbol \mu,{\bf v})$ collected in the next Lemma. 

\begin{lemma} \label{lemma1}
The bilinear form ${\bf a}({\bf u},{\bf v})$ is continuous on ${\bf V} \times {\bf V}$ and coercive, that is
\begin{equation} \label{lemma1:one}
\exists \alpha_1 > 0 : {\bf a}({\bf u},{\bf v}) \le \alpha_1||{\bf u}||_{{\bf V}} ||{\bf v}||_{{\bf V}} ,~~ {\rm{for~all}} ~{\bf u} \in {\bf V},~~ {\rm{for~all}} ~{\bf v} \in {\bf V},
\end{equation}
\begin{equation} \label{lemma1:two}
\exists \alpha_2 > 0 : {\bf a}({\bf u},{\bf u}) \ge \alpha_2 ||{\bf u}||^2_{{\bf V}} ,~~ {\rm{for~all}} ~{\bf u} \in {\bf V}.
\end{equation}
The bilinear form ${\bf b}(\boldsymbol \mu,{\bf v})$ is continuous on $\boldsymbol \Lambda \times {\bf V}$, that is
\begin{equation} \label{lemma1:three}
\exists M > 0 : {\bf b}(\boldsymbol \mu,{\bf v}) \le M ||\boldsymbol \mu||_{\boldsymbol\Lambda} ||{\bf v}||_{{\bf V}} ,~~ {\rm{for~all}} ~\boldsymbol \mu \in \boldsymbol \Lambda,~~ {\rm{for~all}} ~{\bf v} \in {\bf V}
\end{equation}
and it satisfies the {\rm {inf-sup}} condition
\begin{equation} \label{lemma1:four}
\exists \beta>0: \inf_{\boldsymbol \lambda \in \boldsymbol\Lambda}  \sup_{{\bf v} \in {\bf V}} \frac {{\bf b}(\boldsymbol \lambda,{\bf v}) }
{||\boldsymbol \lambda ||_{\boldsymbol \Lambda} ||{\bf v}||_{{\bf V}}} \ge \beta.
\end{equation}
\end{lemma}

\begin{proof}
Properties (\ref{lemma1:one})-(\ref{lemma1:three}) follow easily. It is also the case for the inf-sup condition \eqref{lemma1:four} since, choosing successively ${\bf v}=(v_1,0,0)$ and ${\bf v}=(0,v_2,0)$, we obtain
\begin{eqnarray*}
\sup_{{\bf v} \in {\bf V}} \frac {b(\boldsymbol \lambda,{\bf v}) }
{||{\bf v}||_{{\bf V}}} &=& \sup_{{\bf v} \in {{\bf V}} } \frac {\sum_i  < \lambda_i , v_i-v_{\gamma} >_{\gamma} } 
{ ||{\bf v}||_{{\bf V}}} \geq \frac{1}{2} \Big( \sup_{v_1\in V^1} \frac { < \lambda_1 , v_1 >_{\gamma}  }{||v_1||_{1,\Omega_1} } +  \sup_{v_2\in V^2} \frac { < \lambda_2 , v_2 >_{\gamma}  }{||v_2||_{1,\Omega_2} } \Big) \\
&\geq & \frac{1}{2} \Big( \sup_{v_1\in H^1_{0,\partial \Omega_1 \backslash \gamma} (\Omega_1)} \frac { < \lambda_1 , v_1 >_{\gamma}  }{||v_1||_{1,\Omega_1} } +  \sup_{v_2\in H^1_{0,\partial \Omega_2 \backslash \gamma} (\Omega_2)} \frac { < \lambda_2 , v_2 >_{\gamma}  }{||v_2||_{1,\Omega_2} } \Big) \\
&= &  \frac{1}{2} \Big(  \|\lambda_1 \|_{-1/2, \gamma} + \|\lambda_2 \|_{-1/2, \gamma} \Big) \geq \frac{1}{2}  \|\boldsymbol \lambda \|_{\boldsymbol \Lambda}.
\end{eqnarray*}
\end{proof}

\begin{remark}
A constant $\beta=1$ could be obtained following \cite{BK00} where a constant independent of the number of subdomains is needed. It relies on the definition of a lifting function for functions in $H^{1/2}_{00}(\gamma)$ as introduced in Remark \ref{remark_equivalency}. Since for our application this constant does not play any role, we do not present this computation.
\end{remark}

\begin{theorem}\label{theorem1}
 For $\rho \in L^2(\gamma)$, there exists a unique  solution $({\bf u},\boldsymbol \lambda) \in {\bf V} \times \boldsymbol \Lambda$ to problem \eqref{interf_VF}. Moreover, the following bounds hold
\begin{eqnarray} 
 ||{\bf u}||_{\bf V} &\le& \frac 1 {\alpha_2} ||\rho||_{0,\gamma} , \label{theorem1:one}\\
 \label{theorem1:two}
 ||\boldsymbol \lambda||_{\boldsymbol \Lambda} &\le& \frac{1}{\beta} (1+\frac{\alpha_1}{\alpha_2} ) ||\rho||_{0,\gamma} .
 \end{eqnarray}
\end{theorem}
\begin{proof} Existence and bounds (\ref{theorem1:one})-(\ref{theorem1:two}) follow simply from the general theory \cite{BF} thanks to Lemma \ref{lemma1}. 
\end{proof}

\begin{remark}\label{Remark_inf_sup}     
Notice that Theorem \ref{theorem1} needs only an {\rm inf-sup} condition relating the Lagrange multipliers space $\Lambda^i$ to $V^i$, space of functions on $\Omega_i$. This fact, together with the coercivity of ${\bf a}({\bf u},{\bf v})$ on the whole ${\bf V}$, will allow us to choose the finite dimensional subspace of $V^{\gamma}$ independently of the other spaces.
\end{remark}


\section{Finite element approximation}\label{discrete_form}

We introduce here a Galerkin discretization of problem \eqref{interf_VF}, with care in the need of compatibility between the discrete approximations of $V^i$ and $\Lambda^i$. Let $\{ \mathcal{T}_{h_i}(\Omega_i)\}_{h_{i}}$ be a shape regular family of decompositions of $\Omega_i$ into triangles  and $\{ \mathcal{E}_{h_\gamma}(\gamma)\}_{h_{\gamma}}$ be a regular family of decompositions of $\gamma$ into intervals. Moreover, we denote by $\{ \mathcal{T}_{h_i}(\gamma)\}_{h_{i}}$ the family of decompositions of $\gamma$ induced by $\{ \mathcal{T}_{h_i}(\Omega_i)\}_{h_{i}}$.

\noindent The discrete spaces for the potential in the domains $\Omega_i$'s and on $\gamma$ are chosen as follows:

$V_{h_i}^i=\{ v \in \mathcal{C}^0(\overline{\Omega_i}), v|_{T} \in \mathbb{P}^1(T) \hbox{ for all } T \in \mathcal{T}_{h_i}(\Omega_i), v=0 \hbox{ on } \Gamma_D^i\}$,

$V^{\gamma}_{h_{\gamma}} =\{ v\in \mathcal{C}^0(\overline {\gamma} ), v|_{e} \in \mathbb{P}^1(e), \hbox{ for all } e \in  \mathcal{E}_{h_\gamma}(\gamma), v(0)=v(L)=0\}$,

 ${\bf V}_h=V_{h_1}^1 \times V_{h_2}^2 \times V^{\gamma}_{h_{\gamma}}$  $\subset$  ${\bf V}$.
 
\noindent On the interface $\gamma$ the discrete spaces for the Lagrange multipliers are made of linear functions on the intervals $e \in \mathcal{T}_{h_i}(\gamma)$, modified to be constant in the limit intervals.  Therefore, denoting by $e_{0}^i$ and $e_L^i$ the first and the last interval of $\mathcal{T}_{h_i}(\gamma)$, we introduce

$\Lambda_{h_i}^i =\{ \lambda\in \mathcal{C}^0(\overline{\gamma} ), \lambda|_e \in \mathbb{P}^0(e) \hbox{ for all } e \in \{e_0^i,e_L^i\}, \ \ \lambda|_e \in \mathbb{P}^1(e) \hbox{ for all } e \in \mathcal{T}_{h_i}(\gamma)\backslash \{e_0^i,e_L^i\} \} $,

$\boldsymbol \Lambda_h=\Lambda_{h_1}^1 \times \Lambda_{h_2}^2  \subset \boldsymbol \Lambda$.

\noindent Since the elements of $\Lambda_{h_i}^i$ are p.w. polynomials, the duality pairing is an integral and we can write, for $\lambda_{h_i}^i \in \Lambda_{h_i}^i$, 
\begin{equation}\label{duality_product_discr}
< \lambda^i_{h_i}, v_i >_{\gamma}=\int_{\gamma} \lambda_{h_i}^i v_i \, dx, \quad  \forall v_i \in V^i.
\end{equation}
In the following, to simplify the presentation, we use the notation $h$ instead of $h_i$, unless it might create some confusion.

\subsection{Discrete problems}

\noindent The discrete variational problem corresponding to  \eqref{interf_VF} is\\
{\bf Discrete variational formulation:}

\noindent  Find $({\bf u}_h,\boldsymbol \lambda_h) \in {\bf V}_h\times \boldsymbol \Lambda_h$ s.t. 
\begin{equation}\label{interf_discrete_VF}
\left \{
\begin{array}{r c l @{\hspace{0.3cm}} l}
{\bf a}({\bf u}_h,{\bf v}_h) - {\bf b}(\boldsymbol \lambda_h,{\bf v}_h) &=& \int_{\gamma} \rho \, v^{\gamma}_{h_{\gamma}}\, dx,  &\forall {\bf v}_h \in {\bf V}_h,\\ 
{\bf b}(\boldsymbol \mu_h,{\bf u}_h) &=& 0, &\forall \boldsymbol\mu_h \in \boldsymbol \Lambda_h.
\end{array}
\right. 
\end{equation}
First of all we want to enlighten the  interface structure of formulation \eqref{interf_discrete_VF}, starting from its algebraic form. We introduce the following bilinear forms
\begin{equation*}
\begin{array}{rclcrcl}
 a_i(u_h^i,v_h^i)&=& \int_{\Omega_i} \epsilon_{ox}  \nabla u_h^i \cdot \nabla v_h^i  \, dx dy,\hspace{0.5cm} &&a_{\gamma}(u_{h_{\gamma}}^{\gamma},v_{h_{\gamma}}^{\gamma})&=& d \int_{\gamma} \epsilon_{ch} ( u_{h_{\gamma}}^{\gamma})' ( v_{h_{\gamma}}^{\gamma})' \, dx  , \vspace{0.3cm}\\  
 b_i(\mu^i_h,u^i_h)&=& \int_{\gamma} \mu_{h}^i u_h^i \, dx,\hspace{0.5cm}
&&b^i_{\gamma}(\mu^i_h,u^{\gamma}_{h_{\gamma}})&=& \int_{\gamma} \mu_{h}^i u^{\gamma}_{h_{\gamma}} \, dx,
\end{array}
\end{equation*}
with $i=1,2$. Then, problem \eqref{interf_discrete_VF} can be written in matrix form as follows
\begin{equation}
 \begin{bmatrix}
  \mathbb{ A}_1 & 0 & -\mathbb{B}_1^T & 0 & 0\\
 0 & \mathbb{ A}_2 & 0 & -\mathbb{B}_2^T& 0 \\
   \mathbb{B}_1 & 0 &0& 0 & -\mathbb{B}^1_{\gamma}\\
  0 & 0 & \mathbb{B}_2 & 0  &- \mathbb{B}^2_{\gamma} \\
 0& 0 & ( \mathbb{B}^1_{\gamma})^T  & (\mathbb{B}^2_{\gamma})^T & \mathbb{A}_{\gamma}
 \end{bmatrix}
\begin{bmatrix}
\underline{u}_1 \\
\underline{u}_2 \\
\underline{\lambda}_1 \\
\underline{\lambda}_2 \\
\underline{u}_{\gamma}
 \end{bmatrix}
 =
 \begin{bmatrix}
  0 \\
  0\\
  0 \\
  0\\
  \underline{r}
 \end{bmatrix},
\end{equation}
where $\underline{u}_i$, $\underline{\lambda}_i$, $\underline{u}_{\gamma}$ denote the unknown coefficient vectors of $u_h^i$, $\lambda_h^i$, $u_{h_{\gamma}}^{\gamma}$, respectively, $\mathbb{A}_i$, $\mathbb{A}_{\gamma}$, $\mathbb{B}_i$, $\mathbb{B}_{\gamma}^i$ correspond to the different bilinear forms defined above, and $  \underline{r}$ is the vector corresponding to the right hand side.

Since $\mathbb{A}_i$'s are obviously positive definite, we can first eliminate the unknown vectors $\underline{u}_i$'s. Then, we can eliminate the unknown vectors $\underline{\lambda}_i$'s since it is easy to check that $\hbox{Ker} \ \mathbb{B}_i^T ={0} $. It leads to the following linear system acting only on the unknown vector $\underline{u}_{\gamma}$
\begin{equation}\label{shur}
 \Big ( \mathbb{A}_{\gamma} + \sum_{i=1}^2 (\mathbb{B}^i_{\gamma})^T (\mathbb{B}_i \mathbb{ A}_i^{-1} \mathbb{B}_i^T )^{-1}  \mathbb{B}^i_{\gamma} \Big ) \, \underline{u}_{\gamma} \, = \, \underline{r} .
\end{equation}
\medskip\noindent
We can reinterpret \eqref{shur} in terms of a bilinear form acting on $V^{\gamma}_{h_{\gamma}} \times V^{\gamma}_{h_{\gamma}}$. It is not difficult to see that, starting from the first equation of \eqref{interf_discrete_VF}, with ${\bf v}_h = (0,0,v^{\gamma}_{h_{\gamma}})$, we obtain the following problem:\\
Find $u^{\gamma}_{h_{\gamma}} \in V^{\gamma}_{h_{\gamma}}$ s.t.
\begin{equation}\label{interface}
 a_{\gamma} (u^{\gamma}_{h_{\gamma}},v^{\gamma}_{h_{\gamma}}) + \sum_{i=1}^2 b^i_{\gamma}(\lambda^i_h(u^{\gamma}_{h_{\gamma}}),v^{\gamma}_{h_{\gamma}}) = \int_{\gamma} \rho v^{\gamma}_{h_{\gamma}} \, dx \qquad \forall v^{\gamma}_{h_{\gamma}} \in V^{\gamma}_{h_{\gamma}}.
\end{equation}
For a given $u^{\gamma}_{h_{\gamma}}$, $\lambda^i_h(u^{\gamma}_{h_{\gamma}})$ in \eqref{interface} is the second component of the solution to the following 2D problem :\\
\noindent Find $(u_h^i(u^{\gamma}_{h_{\gamma}}),\lambda^i_h(u^{\gamma}_{h_{\gamma}})) \in V_h^i \times \Lambda_h^i$ s.t.
\begin{equation}\label{bulk}
\left \{
\begin{array}{r c l @{\hspace{0.3cm}} l}
a_i (u_h^i(u^{\gamma}_{h_{\gamma}}),v_h^i) -   b_i(\lambda^i_h(u^{\gamma}_{h_{\gamma}}),v^i_h) &=& 0  &\qquad \forall v^i_h \in V^i_h, \\
 b_i(\mu^i_h,u_h^i(u^{\gamma}_{h_{\gamma}})) &=& b^i_{\gamma}(\mu^i_h,u^{\gamma}_{h_{\gamma}})  
&\qquad \forall \mu_h^i \in \Lambda_h^i .
\end{array}
\right. 
\end{equation}
The matrix form of \eqref{interface} (with \eqref{bulk}) is indeed \eqref{shur}.

We see that in order to solve the interface problem \eqref{interface} (which is a well posed problem for a given $\lambda^i_h$), we need to solve the saddle point problems \eqref{bulk} for $i=1,2$, for which an inf-sup condition linking only $V_h^i$ and $\Lambda^i_h$ is required (without connection with $V_{\gamma}$). It was already anticipated in Remark \ref{Remark_inf_sup} and it will be explicit in the next subsection, where the analysis of  formulation \eqref{interf_discrete_VF} is done.

\begin{remark}
Notice that the 1D interface problem \eqref{interface} includes (through $\lambda^i_h(u^{\gamma}_{h_{\gamma}})$ ) a discrete Poincar\'e-Steklov operator that maps the Dirichlet datum $u^{\gamma}_{h_{\gamma}}$ on $\gamma$ into $\lambda^i_h$ which, as explained in \eqref{conormal} at the continuous level, is linked to the conormal derivative of $u^i_h$ on $\gamma$. Observing that the Poincar\'e-Steklov operator is an operator of order 1, with a strict analogy to the 1/2-Laplacian, it is related to \cite{EHM14}, where the specificity of the surface particle density leads to construct the self-consistent potential for a graphene sheet as the solution on the plane of a fractional Laplacian. \end{remark}

\subsection{Well-posedness and error estimates}

To prove existence and uniqueness of a solution to \eqref{interf_discrete_VF} as stated by Theorem \ref{Theorem_discrete}, we resort to Fortin's argument. More precisely, in order to obtain the discrete inf-sup condition (written later on in \eqref{discr_inf_sup}) as well as some error estimates,
 we introduce in the following Lemma a projector $\pi_h^i:$ $L^2(\gamma)$ 
 $\longrightarrow W^i_h$ for the interface functions, where 
$$W_{h}^i = V_{h}^i|_{\gamma}$$
is the trace space of $V^i_{h}$ on $\gamma$. Notice that, due to the homogeneous conditions on $\Gamma_D^i$ in the definition of $V_h^i$, a function $ w_h^i \in W_h^i$ satisfies $w_h^i(0)=w_h^i(L)=0$. So the dimension of $W_h^i$ is equal to the number of intervals of $ \mathcal{T}_{h_i}(\gamma)$ minus 1, matching the dimension of $\Lambda^i_h$.

\begin{lemma} \label{lemma_pi_gam}
There exists $\pi_h^i:$ $L^2(\gamma)$ 
 $\longrightarrow W^i_h $ defined by 
 \begin{equation} \label{pi_gam_orth}
 \int_{\gamma} \lambda^i_h (\pi^i_h \eta- \eta)\, dx =0, \quad \forall \lambda^i_h \in \Lambda^i_h   ,
 \end{equation} 
 such that, for all $\eta \in H^{1/2}_{00}(\gamma)$, 
 \begin{equation}
 ||\pi^i_h \eta||_{1/2,\gamma} \le C  || \eta||_{1/2,\gamma}. \label{pi_gam_bound}
 \end{equation}
 \end{lemma}
 
\begin{proof}
For any $w_h^i \in W_h^i$, there exists $\overline{\lambda_h^i} (w_h^i) \in \Lambda_h^i$, such that
\begin{equation} \label{bb}
 \int_{\gamma} \overline{\lambda^i_h}(w_h^i) w_h^i\, dx \ge ||w^i_h||_{0,\gamma}^2,
\end{equation}
which implies uniqueness (and consequently existence) of the linear operator $\pi_h^i$. Indeed, by choosing $\overline{\lambda_h^i} (w_h^i){|_e}=w_h^i{|_e}$ for $e \in \mathcal{T}_{h_i}(\gamma)\backslash \{e_0^i,e_L^i\}$, noting that the constant value $\overline{\lambda_h^i} (w_h^i){|_{e^i_0}}$ coincides with $w_h^i$ evaluated in the end point of $e^i_0$ (and analogously for $\overline{\lambda_h^i} (w_h^i){|_{e^i_L}}$), we can easily obtain \eqref{bb} and the following bound
 \begin{equation} \label{cc}
  ||\overline{\lambda^i_h}(w_h^i)||_{0,\gamma} \le  C ||w^i_h||_{0,\gamma}.
 \end{equation}
 Then, \eqref{bb} and \eqref{cc} give 
\begin{equation} \label{coerc_bound1}
 \sup_{\lambda^i_h \in \Lambda^i_h} \frac {\int_{\gamma} \lambda^i_h w_h^i\, dx}{||\lambda^i_h||_{0,\gamma}} \ge \frac 1 C ||w^i_h||_{0,\gamma}.
\end{equation}
Moreover, by using \eqref{pi_gam_orth}, for each $\eta_h \in W_h^i$, \eqref{coerc_bound1}  gives
\begin{equation*} \label{err_eta}
 ||\eta_h - \pi_h^i \eta||_{0,\gamma} \le C \sup_{\lambda^i_h \in \Lambda^i_h} \frac {\int_{\gamma} \lambda^i_h ( \eta_h - \pi_h^i \eta)\, dx}{||\lambda^i_h||_{0,\gamma}} = C \sup_{\lambda^i_h \in \Lambda^i_h} \frac {\int_{\gamma} \lambda^i_h (\eta_h - \eta)\, dx}{||\lambda^i_h||_{0,\gamma}} \le C ||\eta - \eta_h||_{0,\gamma}.
\end{equation*}
Consequently, using a triangular inequality and classical approximation results, see \cite{Ci02} e.g., we obtain
\begin{equation} \label{bound:1/2}
 ||\eta -  \pi_h^i\eta||_{0,\gamma} \le C \inf_{\eta_h \in W^i_h} ||\eta - \eta_h||_{0,\gamma} \le C h_i^{1/2} ||\eta||_{1/2,\gamma}.
\end{equation}

Next, we can obtain the uniform bound  \eqref{pi_gam_bound} with the same argument as in \cite[Lemma 2]{B03}. For completeness, we give a sketch of the proof. We introduce $\widehat \pi_h^i:$ $H^{1/2}_{00}(\gamma)$ 
 $\longrightarrow W^i_h $ as the projection induced by the $H^{1/2}_{00}(\gamma)$ norm, defined by
 $$
 ||\widehat \pi_h^i \eta - \eta||_{1/2,\gamma} = \inf_{\eta_h \in W_h^i} ||\eta - \eta_h||_{1/2,\gamma}, \qquad \eta \in H^{1/2}_{00}(\gamma).
 $$
 Since trivially $||\widehat \pi_h^i \eta ||_{1/2,\gamma} \le C || \eta||_{1/2,\gamma}$, by applying triangular inequalities and a classical inverse inequality, we obtain
 \begin{eqnarray} \label{bound:pi}
  ||\pi_h^i \eta ||_{1/2,\gamma}   &\le &C || \eta||_{1/2,\gamma} + h_i^{-1/2} (||\pi_h^i \eta -  \eta||_{0,\gamma}  +||\eta - \widehat \pi_h^i \eta||_{0,\gamma}).
 \end{eqnarray}
 Using classical duality arguments, it is possible to prove that 
 \begin{equation} \label{bound:pihat}
 ||\eta - \widehat \pi_h^i \eta||_{0,\gamma} \le C h_i^{1/2} ||\eta ||_{1/2,\gamma}.
 \end{equation}
Indeed, let $\chi(f) \in H^{1/2}_{00} (\gamma)$ be the solution of the problem
 \begin{equation} \label{pb:dual}
  (\chi(f), \xi)_{1/2,\gamma} = \int_{\gamma} f \xi dx \qquad \forall \xi \in H^{1/2}_{00}(\gamma).
 \end{equation}
If $f \in L^2(\gamma)$ then $\chi(f) \in H^1_0(\gamma)$ and the following bound holds
\begin{equation} \label{bound:chi}
||\chi(f)||_{1,\gamma} \le ||f||_{0,\gamma}.
\end{equation}
Since $(\chi_h, \eta - \widehat \pi_h^i \eta)_{1/2,\gamma} =0$ for all $\chi_h \in W^i_h$, starting from  \eqref{pb:dual} with $f=\xi=\eta - \widehat \pi_h^i \eta$, we obtain
\begin{eqnarray*}
 ||\eta - \widehat \pi_h^i \eta||_{0,\gamma}^2 = (\chi(\eta - \widehat \pi_h^i \eta), \eta - \widehat \pi_h^i \eta)_{1/2,\gamma} &=&  \inf_{\chi_h \in W^i_h}  (\chi(\eta - \widehat \pi_h^i \eta) - \chi_h, \eta - \widehat \pi_h^i \eta)_{1/2,\gamma}.
 \end{eqnarray*}
 Then, using a classical approximation result, it gives the following estimate
 \begin{eqnarray*}
 ||\eta - \widehat \pi_h^i \eta||_{0,\gamma}^2 &\le& \inf_{\chi_h \in W^i_h} ||\chi(\eta - \widehat \pi_h^i \eta) -\chi_h ||_{1/2,\gamma}   ||\eta - \widehat \pi_h^i \eta||_{1/2,\gamma} \\
 & \le & C h_i^{1/2} ||\chi(\eta - \widehat \pi_h^i \eta)||_{1,\gamma} ||\eta - \widehat \pi_h^i \eta||_{1/2,\gamma} .
\end{eqnarray*}
Thus \eqref{bound:pihat} follows, thanks to \eqref{bound:chi} and to the boundedness of $||\widehat \pi_h^i \eta||_{1/2,\gamma}$.
 
Finally, using \eqref{bound:1/2} and \eqref{bound:pihat} in \eqref{bound:pi}, we obtain the bound \eqref{pi_gam_bound} and conclude the proof.

\end{proof}

\begin{remark}
A projector with similar properties could be also obtained for a more general definition of $\Lambda_h^i$ (and of $W_h^i$). For instance, $\Lambda_h^i$ could be defined on a decomposition  $\mathcal{E}_{h_{\lambda}}(\gamma)$ different from $\mathcal{T}_{h_i}(\gamma)$, with the only requirement that all nodes of $\mathcal{E}_{h_{\lambda}}(\gamma)$ are also nodes of $\mathcal{T}_{h_i}(\gamma)$ and $W_h^i$ could be a subspace of the trace space $V_{h_i}^i|_{\gamma}$ defined on the same decomposition $\mathcal{E}_{h_{\lambda}}(\gamma)$. Since, for our application, $\lambda_h^i$ plays only the role of a working quantity, we do not discuss this issue further.
\end{remark}

\begin{theorem}\label{Theorem_discrete}
 For $\rho \in L^2(\gamma)$, there exists a unique solution $({\bf u}_h,\boldsymbol \lambda_h) \in {\bf V}_h \times \boldsymbol \Lambda_h$ to problem \eqref{interf_discrete_VF}.  Moreover, the following error estimate holds
 \begin{equation}\label{error_bound}
|| {\bf u}-{\bf u}_h||_{\bf V} + ||\boldsymbol\lambda - \boldsymbol\lambda_h||_{\boldsymbol \Lambda} \le C \big( \inf_{{\bf v}_h \in {\bf V}_h} ||{\bf u}-{\bf v}_h||_{\bf V} + \inf_{\boldsymbol\mu_h \in \boldsymbol \Lambda_h} ||\boldsymbol\lambda - \boldsymbol\mu_h||_{\boldsymbol \Lambda} \big) .
\end{equation} 
 \end{theorem}
 \begin{proof}
Since the embedding ${\bf V}_h\subset {\bf V}$ implies immediately the coerciveness 
 \begin{equation}\label{coerc_h}
 {\mathbf a}({\bf u}_h,{\bf u}_h) \ge \alpha_2 ||{\bf u}_h||^2_{{\bf V}} \quad \forall {\bf u}_h \in {\bf V}_h,
 \end{equation}
we only have to prove the discrete inf-sup condition
\begin{equation}\label{discr_inf_sup}
\exists \delta>0, \quad \inf_{\boldsymbol \lambda_h \in \boldsymbol \Lambda_h} \sup_{{\bf v}_h\in {\bf V}_h} \frac{{\bf b}(\boldsymbol \lambda_h,{\bf v}_h)}{\| {\bf v_h} \|_{\bf V} \| \boldsymbol\lambda_h\|_{\boldsymbol \Lambda}} \geq \delta .
\end{equation}
As for the continuous case, it is enough to show that
it exists $\delta>0$ such that, for all $\lambda^i_h \in \Lambda^i_h$,
\begin{equation}\label{ineq_lambda_i_discr} 
||\lambda^i_h ||_{-1/2, \gamma} \leq \delta \sup_{v^i_h\in V^i_h} \frac { \int_{\gamma} \lambda^i_h v^i_h dx  }{||v^i_h||_{1,\Omega_i} },
\end{equation}
where we used \eqref{duality_product_discr}. This inequality  relies on the existence of a bounded projector 
$\Pi_h^i:$ $V^i
 \longrightarrow V^i_h $ such that, for all $v_i \in V^i$,
 \begin{eqnarray}
\int_{\gamma} \lambda^i_h (\Pi^i_h v_i - v_i ) dx &=&0 \quad \forall \lambda^i_h \in \Lambda^i_h \label{Pi_orth} ,\\
 ||\Pi^i_h v_i||_{1,\Omega_i} &\le& \kappa  || v_i||_{1,\Omega_i} \label{Pi_bound},
 \end{eqnarray}
 with $\kappa >0$  constant independent of $h$. Given $v_i \in V^i$, let us define $\Pi^i_h v_i$ in $V_h^i$ as the discrete lifting of $\pi^i_h v_i$, that is  
 \begin{eqnarray*}
a_i(\Pi^i_h v_i,v_h^i) &=&0 \quad \forall v_h^i \in V_h^i ,\\
( \Pi^i_h v_i )|_{\gamma} &=& \pi^i_h (v_i |_{\gamma} ). 
 \end{eqnarray*}
We emphasize that, due to the homogeneous boundary condition on $\Gamma_D^i$, such a lifting exists thanks to the null value of
$\pi^i_h v_i$ at the end points of $\gamma$.
\eqref{Pi_orth} is trivially satisfied because of \eqref{pi_gam_orth}. The bound \eqref{Pi_bound} is obtain combining the classical result $||\Pi^i_h v_i||_{1,\Omega_i} \le C ||\pi_h^i v_i ||_{1/2,\gamma}$ (that can be found e.g. in \cite[Lemma 3.2.]{BPS86}), the bound \eqref{pi_gam_bound} that holds since $v_i |_{\gamma} \in H^{1/2}_{00}(\gamma)$ and the trace theorem.

\noindent Now, since $\Lambda^i_h \subset \Lambda_i$, we have 
\begin{equation*}
||\lambda^i_h ||_{-1/2, \gamma} \leq \sup_{v_i\in V^i} \frac { \int_{\gamma} \lambda^i_h  v_i  dx   }{||v_i||_{1,\Omega_i} }.
\end{equation*}
Using \eqref{Pi_orth} and \eqref{Pi_bound} , we obtain that, for all  $v_i\in V^i$, 
\begin{equation*}
 \frac { \int_{\gamma} \lambda^i_h  v_i  dx}  {||v_i||_{1,\Omega_i} } \leq \kappa \frac { \int_{\gamma} \lambda^i_h  \Pi_h^i v_i  dx  }{||\Pi_h^i v_i||_{1,\Omega_i} } \leq  \kappa \sup_{v^i_h\in V^i_h} \frac { \int_{\gamma}  \lambda^i_h  v^i_h dx }{||v^i_h||_{1,\Omega_i} }. 
 \end{equation*}
 Thus, \eqref{ineq_lambda_i_discr}  follows with $\delta =\kappa$.
The discrete inf-sup condition \eqref{discr_inf_sup} and the coerciviness \eqref{coerc_h} imply existence and uniqueness of the discrete solution as well as the error bound \eqref{error_bound}.
\end{proof}

\begin{remark}
As for the continuous case and as suggested by the formulation \eqref{interface}-\eqref{bulk}, the discrete inf-sup condition \eqref{discr_inf_sup} relates only the Lagrange multipliers space $\Lambda^i_h$ to the space $V^i_h$. It allows us to choose $V^{\gamma}_{h_{\gamma}}$ independently of the other spaces. This point is particularly interesting when the density $\rho$ appearing in the second member of \eqref{interf_discrete_VF} requires a refinement of the interface discretization.
\end{remark}

\noindent An optimal order of convergence is obtained for a regular $u_i$ as in the next theorem.

\begin{theorem}\label{Th_final_err_est}
Assume  $u_i \in H^2(\Omega_i)$, $i=1,2$. Then, 
$$ ||{\bf u}-{\bf u}_h||_{\bf V} + ||\boldsymbol \lambda - \boldsymbol\lambda_h||_{\boldsymbol \Lambda} \le C \Big(\sum_i h_i ^2 ||u_i||^2_{2,\Omega_i} + h_{\gamma}^2 ||u_\gamma||^2_{2,\gamma} \Big)^{1/2}.
 $$
\end{theorem}
 
\begin{proof}
If $u_i \in H^2(\Omega_i)$ then $\epsilon_{ox} \nabla u_i \cdot n_i \in H^{1/2}(\gamma)$ and  $u_{\gamma}\in H^2(\gamma)$. Let us introduce the Lagrange interpolation operators $\mathcal{I}^i_{h_i}:$ $V^i\longrightarrow V_{h_i}^i$ and $\mathcal{I}^\gamma_{h_{\gamma}}:$ $V^\gamma\longrightarrow V_{h_{\gamma}}^\gamma$, and the $L^2$-projector $\mathcal{P}^i_{h_i}:$ $L^2(\gamma) \longrightarrow \Lambda_{h_i}^i$, for which the following bounds are classical
$$ \| u_i - \mathcal{I}^i_{h_i} u_i \|_{1,\Omega_i} \leq C h_i ||u_i||_{2,\Omega_i},$$
$$ \| u_\gamma - \mathcal{I}^\gamma_{h_{\gamma}} u_\gamma \|_{1,\gamma} \leq C h_{\gamma} ||u_\gamma||_{2,\gamma},$$
$$ \| \lambda_i - \mathcal{P}^i_{h_i} \lambda_i \|_{0,\gamma} \leq C h_i^{1/2} \|\lambda_i\|_{1/2,\gamma}.$$
For completing the bound for the ${\boldsymbol \lambda}$ component of the error we again follow \cite{B03}.  Due to the regularity of $\epsilon_{ox} \nabla u_i \cdot n_i $, equation \eqref{conormal} becomes, for $v_i \in V^i$,
\begin{equation*}
<\lambda_i,v_i>_{\gamma} =\int_{\gamma} \epsilon_{ox} \nabla u_i \cdot n_i \, v_i \, dx .
\end{equation*}
Thus, we can write
\begin{equation*}
\| \lambda_i \|_{-1/2,\gamma} \leq \sup_{ v_i \in V^i} \frac {\int_{\gamma} \epsilon_{ox} \nabla u_i \cdot n_i \, v_i \, dx} {||v_i||
_{1,\Omega_i}}.
\end{equation*}
It gives
\begin{eqnarray*}
\inf_{\mu_{h_i}^i \in \Lambda^i_{h_i}} \| \lambda_i - \mu_{h_i}^i\|_{-1/2,\gamma} &\leq&  \sup_{{v_i \in V^i}} \frac {\int_{\gamma} \big(\epsilon_{ox} \nabla u_i \cdot n_i-\mathcal{P}^i_{h_i} (\epsilon_{ox} \nabla u_i \cdot n_i) \big) v_i \, dx} {||v_i||_{1,\Omega_i}} \\
&=& \sup_{{v_i \in V^i}} \frac {\int_{\gamma} \big(\epsilon_{ox} \nabla u_i \cdot n_i-\mathcal{P}^i_{h_i} (\epsilon_{ox} \nabla u_i \cdot n_i) \big) (v_i - \mathcal{P}^i_{h_i}( v_i) ) \, dx} {||v_i||_{1,\Omega_i}} \\
&\leq& h_i \|\epsilon_{ox} \nabla u_i \cdot n_i \|_{1/2,\gamma} \sup_{v_i \in V^i}  \frac { \|v_i\|_{1/2,\gamma}} {||v_i||_{1,\Omega_i}} \leq  C h_i  \|u_i\|_{2,\Omega_i}. 
\end{eqnarray*}
Recalling then \eqref{error_bound}, the desired estimate easily follows.
\end{proof}


\section{Robin type condition at interface}\label{section_robin}

We now consider the Robin type condition introduced in \eqref{eq_continuity2} rather than the condition \eqref{eq_continuity}. It is especially interesting to tackle an anisotropic channel permittivity given by the diagonal tensor \eqref{perm_tensor}.  With homogeneous Dirichlet conditions on $\Gamma_D$, the problem writes\\
{\it Find $(u_1,u_2,u_{\gamma})$ s.t.}
\begin{equation}\label{interf_pb_value_pb_robin}
\left \{
\begin{array}{r c l @{\hspace{1cm}} l}
- \nabla \cdot  (\epsilon_{ox}  \nabla u_i) &=& 0  &\hbox{in}~ \Omega_i, ~~i=1,2,\\
- d (\epsilon_{//} u_{\gamma}' )'  &=&  \rho -  \epsilon_{ox}(\nabla u_1\cdot n_1 +  \nabla  u_2 \cdot n_2)  & \hbox{on}~ \gamma,\\
(u_i-u_{\gamma})+\alpha \ \epsilon_{ox} \nabla u_i \cdot n_i &=& 0, &\hbox{on} ~\gamma,\\
u_{i}&=&0 &\hbox{ on } \Gamma^i_D, \\
\epsilon_{ox}\nabla u_{i} \cdot n_i &=&0 &\hbox{ on } \Gamma^i_N,\\
u_{\gamma}(0)=u_{\gamma}(L)&=&0.
\end{array}
\right. 
\end{equation}
We introduce the space $\boldsymbol Q=\big(L^2 (\gamma)\big)^2 \subset \boldsymbol \Lambda$ and we consider the following variational problem:\\
{\bf Variational formulation - Robin type condition:}\\
\noindent  Find $({\bf u},\boldsymbol \lambda) \in {\bf V} \times \boldsymbol Q$ s.t.
\begin{equation}\label{interf_VF_robin}
\left \{
\begin{array}{r  c c c l @{\hspace{1cm}} l}
{\bf a}({\bf u},{\bf v}) &-& {\bf b}(\boldsymbol \lambda,{\bf v}) &=& \int_{\gamma} \rho \, v_{\gamma}\, dx,  & \forall {\bf v} \in {\bf V}, \\ 
{\bf b}(\boldsymbol \mu,{\bf u}) &+& \alpha \  {\bf c}(\boldsymbol \lambda,\boldsymbol  \mu) &=& 0, & \forall \boldsymbol \mu \in {\bf Q},
\end{array}
\right. 
\end{equation}
where ${\bf c}$ is the bilinear form defined by
$${\bf c}(\boldsymbol \lambda,\boldsymbol  \mu) = \sum_{i=1}^2 \int_{\gamma} \lambda_i \mu_i \ dx.$$
Notice that due to the bilinear form ${\bf c}$, this formulation requires more regularity on the Lagrange multipliers $\boldsymbol \lambda$, compared to \eqref{interf_VF}.

The problem is coercive on $\boldsymbol V \times \boldsymbol Q$, so we easily have existence and uniqueness of the solution by Lax-Milgram theorem, with the following  bound
\begin{equation*}
\|{\bf u}  \|_{\bf V}^2 + \alpha \| \boldsymbol \lambda\|^2_{\bf Q} \leq C \|\rho\|_{0,\gamma}^2.
\end{equation*}
In order to retrieve a control on $\boldsymbol \lambda$ independent on $\alpha$ (and thus on the effective dielectric thickness $d$), following \cite{BF}, we can use the inf-sup condition on the largest space $\boldsymbol \Lambda$ introduced in \eqref{lemma1:four}. Indeed,
\begin{eqnarray*}
\beta \|\boldsymbol \lambda \|_{\boldsymbol \Lambda} &\leq& \sup_{{\bf v} \in {\bf V}} \frac{{\bf b}(\boldsymbol \lambda,{\bf v} )}{||{\bf v}||_{\bf V}}
= \sup_{{\bf v} \in {\bf V}} \frac{{\bf a}({\bf u},{\bf v}) -  \int_{\gamma} \rho \, v_{\gamma}\, dx}{||{\bf v}||_{{\bf V}}} \leq C ( ||{\bf u}||_{\bf V} +  \|\rho\|_{0,\gamma} ).
\end{eqnarray*}
Putting the two estimates together, we obtain the result summarized in the following theorem:

\begin{theorem}
 For $\rho \in L^2(\gamma)$, there exists a unique  solution $({\bf u},\boldsymbol \lambda) \in {\bf V} \times \boldsymbol Q$ to problem \eqref{interf_VF_robin}. Moreover, the following bound holds
\begin{equation}
 \|{\bf u}  \|_{\bf V}^2 + \|\boldsymbol \lambda \|_{\boldsymbol \Lambda} ^2+ \alpha \| \boldsymbol \lambda\|^2_{\bf Q} \leq C \|\rho\|_{0,\gamma}^2.
\end{equation}
\end{theorem}

At the discrete level, we notice that $\boldsymbol \Lambda_h$ defined in Section \ref{discrete_form} is contained in ${\bf Q}$. Consequently, a discrete variational problem associated to \eqref{interf_VF_robin} is \\
{\bf Discrete variational formulation - Robin type condition:}

\noindent  Find $({\bf u}_h,\boldsymbol \lambda_h) \in {\bf V}_h\times \boldsymbol \Lambda_h$ s.t.
\begin{equation}\label{interf_discrete_VF_robin}
\left \{
\begin{array}{r c c cl @{\hspace{0.2cm}} l}
{\bf a}({\bf u}_h,{\bf v}_h) &-& {\bf b}(\boldsymbol \lambda_h,{\bf v}_h) &=& \int_{\gamma} \rho \, v^{\gamma}_{h_{\gamma}}\, dx,  &\forall {\bf v}_h \in {\bf V}_h,\\ 
{\bf b}(\boldsymbol \mu_h,{\bf u}_h) &+& \alpha \  {\bf c}(\boldsymbol \lambda_h,\boldsymbol  \mu_h) &=& 0, &\forall \boldsymbol\mu_h \in \boldsymbol \Lambda_h.
\end{array}
\right. 
\end{equation}
Using the coerciveness of ${\bf a}$ in ${\bf V}$ and ${\bf c}$ in $\boldsymbol Q$ and the continuity bounds of the three bilinear forms, we obtain after standard computations
$$
 \|{\bf u} - {\bf u}_h  \|_{\bf V}^2 + \alpha \| \boldsymbol \lambda - \boldsymbol \lambda_h \|^2_{\bf Q} \leq C ( \inf_{{\bf v}_h \in {\bf V}_h} \| {\bf u} - {\bf v}_h  \|_{\bf V}^2+ \alpha \inf_{\boldsymbol \mu_h  \in {\boldsymbol \Lambda}_h} \|\boldsymbol \lambda - \boldsymbol \mu_h \|_{\boldsymbol Q} ^2).
$$
Again, an error estimate independent on $\alpha$ can then be obtained using the discrete $\inf$-$\sup$ condition \eqref{discr_inf_sup}. Indeed, for any $\boldsymbol \mu_h \in \boldsymbol \Lambda_h$, we have
\begin{eqnarray*}
\beta \|  \boldsymbol\lambda_h -  \boldsymbol \mu_h\|_{ \boldsymbol \Lambda} \leq \sup_{{\bf v}_h\in {\bf V}_h} \frac{{\bf b}( \boldsymbol \lambda_h- \boldsymbol \mu_h,{\bf v}_h)}{\| {\bf v}_h\|_{\bf V}}
&\leq &  \sup_{{\bf v}_h\in {\bf V}_h} \frac{{\bf b}( \boldsymbol \lambda-  \boldsymbol \mu_h,{\bf v}_h)+{\bf a}({\bf u}_h-{\bf u},{\bf v}_h)} {\| {\bf v}_h\|_{\bf V}}\\
&\leq &  M \| \boldsymbol \lambda-  \boldsymbol \mu_h \|_{ \boldsymbol \Lambda} + \alpha_1 \|{\bf u}-{\bf u}_h \|_{\bf V},
\end{eqnarray*}
where $M$ and $\alpha_1$ are positive constants defined in Lemma \ref{lemma1}. With these estimates, we easily obtain the result summarized in the following theorem:
\begin{theorem}\label{Theorem_discrete_robin}
 For $\rho \in L^2(\gamma)$, there exists a unique solution $({\bf u}_h,\boldsymbol \lambda_h) \in {\bf V}_h \times \boldsymbol \Lambda_h$ to problem \eqref{interf_discrete_VF_robin} and the following error estimate holds
 \begin{eqnarray}\label{error_bound_robin}
|| {\bf u}-{\bf u}_h||^2_{\bf V} +  \| \boldsymbol \lambda - \boldsymbol \lambda_h \|^2_{\boldsymbol \Lambda}+ \alpha ||\boldsymbol\lambda - \boldsymbol\lambda_h||^2_{\boldsymbol Q}  &\nonumber \\
& \hspace{-5cm} \le  \displaystyle C \big( \inf_{{\bf v}_h \in {\bf V}_h} ||{\bf u}-{\bf v}_h||^2_{\bf V} +   \inf_{\boldsymbol \mu_h  \in {\boldsymbol \Lambda}_h} \|\boldsymbol \lambda - \boldsymbol \mu_h \|_{\boldsymbol \Lambda} ^2+ \alpha \inf_{\boldsymbol\mu_h \in \boldsymbol \Lambda_h} ||\boldsymbol\lambda - \boldsymbol\mu_h||^2_{\boldsymbol Q} \big) .
\end{eqnarray} 
 \end{theorem}

\noindent As for the case of the continuity condition \eqref{eq_continuity} discussed in the previous section, we can write the algebraic form of \eqref{interf_discrete_VF_robin},  emphasizing the interface structure of the formulation. We introduce first the bilinear form
\begin{equation*}
c_i(\lambda^i_h,\mu^i_h)= \int_{\gamma} \lambda^i_h \mu_{h}^i  \, dx,
\end{equation*}
as well as the associated matrix $\mathbb{C}_i$, with $i=1,2$. Then, problem \eqref{interf_discrete_VF_robin} can be written in matrix form as follows
\begin{equation}
 \begin{bmatrix}
  \mathbb{ A}_1 & 0 & -\mathbb{B}_1^T & 0 & 0\\
 0 & \mathbb{ A}_2 & 0 & -\mathbb{B}_2^T& 0 \\
   \mathbb{B}_1 & 0 & \alpha \mathbb{C}_1& 0 & -\mathbb{B}^1_{\gamma}\\
  0 & 0 & \mathbb{B}_2 &   \alpha  \mathbb{C}_2  &- \mathbb{B}^2_{\gamma} \\
 0& 0 & ( \mathbb{B}^1_{\gamma})^T  & (\mathbb{B}^2_{\gamma})^T & \mathbb{A}_{\gamma}
 \end{bmatrix}
\begin{bmatrix}
\underline{u}_1 \\
\underline{u}_2 \\
\underline{\lambda}_1  \\
\underline{\lambda}_2\\
\underline{u}_{\gamma}
 \end{bmatrix}
 =
 \begin{bmatrix}
  0 \\
  0\\
  0  \\
  0\\
  \underline{r}
 \end{bmatrix}.
\end{equation}

\noindent Doing the same eliminations than the ones leading to \eqref{shur} and using the fact that $ \mathbb{C}_i$'s are positive definite, we obtain the following linear system acting only on the unknown vector $\underline{u}_{\gamma}$ 
\begin{equation}\label{shur_robin}
 \Big ( \mathbb{A}_{\gamma} + \sum_{i=1}^2 (\mathbb{B}^i_{\gamma})^T (\mathbb{B}_i \mathbb{ A}_i^{-1} \mathbb{B}_i^T +\alpha  \mathbb{C}_i )^{-1}  \mathbb{B}^i_{\gamma} \Big ) \, \underline{u}_{\gamma} \, = \, \underline{r} .
\end{equation}
\medskip\noindent
Again, we can reinterpret \eqref{shur_robin} in terms of a bilinear form acting on $V^{\gamma}_{h_{\gamma}} \times V^{\gamma}_{h_{\gamma}}$. More precisely, $u^{\gamma}_{h_{\gamma}}$ is still solution of \eqref{interface} depending on $\lambda^i_h(u^{\gamma}_{h_{\gamma}})$ that is now the second component of the solution to the following 2D problem:\\
Find $u^{\gamma}_{h_{\gamma}} \in V^{\gamma}_{h_{\gamma}}$ s.t.
\begin{equation}\label{bulk_robin}
\left \{
\begin{array}{r c l @{\hspace{0.3cm}} l}
a_i (u_h^i(u^{\gamma}_{h_{\gamma}}),v_h^i) -   b_i(\lambda^i_h(u^{\gamma}_{h_{\gamma}}),v^i_h) &=& 0  &\qquad \forall v^i_h \in V^i_h, \\
 b_i(\mu^i_h,u_h^i(u^{\gamma}_{h_{\gamma}}))+\alpha c_i(\lambda^i_h(u^{\gamma}_{h_{\gamma}}),\mu^i_h) &=& b^i_{\gamma}(\mu^i_h,u^{\gamma}_{h_{\gamma}})  
&\qquad \forall \mu_h^i \in \Lambda_h^i .
\end{array}
\right.
\end{equation}


\section{Numerical experiments} \label{section_num}

To illustrate the approach, we present some numerical tests for a Graphene Field-Effect Transistor (GFET). Self-consistent simulations for such a structure are performed for instance in \cite{NBCD10,NR21}. Here, we consider a longitudinal length $L=60$ nm and a transversal length $l=4$ nm. In the $y$ direction, it contains a single layer of graphene, characterized by an effective dielectric thickness $d=0.2$ nm and a graphene permittivity constant $\epsilon_{ch}=13.9 \ \epsilon_0$ sandwiched between two layers of dielectric $SiO_2$ ($\epsilon_{ox}=3.9 \ \epsilon_0$) \cite{FVF16}, $\epsilon_0$ being the permittivity in vacuum.  Only in Subsection \ref{section_num_anisotropy}, where the condition \eqref{eq_continuity} is replaced by the Robin condition \eqref{eq_continuity2} (as analyzed in Section \ref{section_robin}), the permittivity constant $\epsilon_{ch}$ is replaced by a diagonal tensor with an in-plane permittivity $\epsilon_{//}=13.9 \ \epsilon_0$ and an out-of-plane permittivity $\epsilon_{\perp}=6.9 \ \epsilon_0$. As proposed in \cite{NBCD10} and schematically represented in Fig.\ref{Fig_device}, the transport direction $x$ is composed of a $20$ nm active zone, with a doping concentration $N^-_{dop}=10^{14}$ m$^{-2}$, sandwiched between a $20$ nm Source region and a $20$ nm Drain region, both highly doped  ($N^+_{dop}=10^{17}$ m$^{-2}$). Source and Drain potentials are imposed on the entire vertical edges $\{ 0\} \times (-\frac{l}{2},\frac{l}{2})$ and $\{L\} \times(-\frac{l}{2},\frac{l}{2})$. Most of the results presented here correspond to thermal equilibrium (zero applied Drain-Source voltage $V_{DS}=V_D-V_S=0.0$ V). For out-of-equilibrium results, we use an iterative process, starting with $V_{DS}=0$ V and then incrementing  $V_{DS}$ with an increment step of $0.01$ V. Finally, a Gate potential $V_G$ is imposed on $\{ \pm \frac{l}{2}\} \times ] x_G,L-x_G[$, with $x_G=10$ nm, to modulate the particle transport. Since the effect of changing the gate voltage is as expected for a double gate device and it does not infer on the interface approach, we only consider here the case $V_G=0$ V.

\begin{figure}[h!]
\begin{center}
\includegraphics[scale=0.6, keepaspectratio=true]{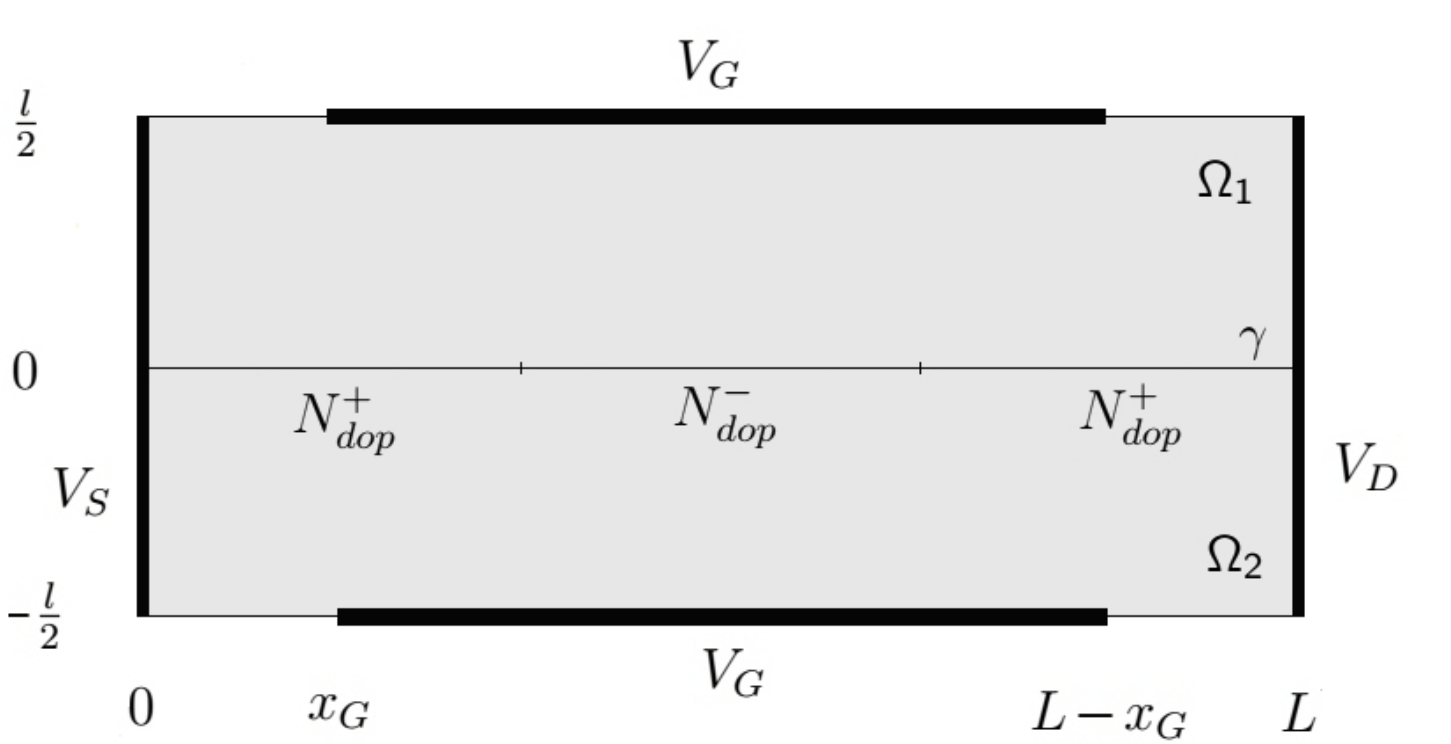}
\caption{Schematic representation of the GFET.}\label{Fig_device}
\end{center}
\end{figure}

In the experiments presented below, we use a regular uniform discretization. More precisely, the triangulation $\{ \mathcal{T}_{h_i}(\Omega_i)\}_{h_{i}}$ is obtained defining a cartesian grid with $N_x$ and $N_y$ edges respectively in the $x$ and in the $y$ direction. We define $h_i$ as the scaled triangle diameter $\sqrt{\frac{1}{N_x^2}+\frac{l^2}{L^2}\frac{1}{N^2_y}}$. Moreover, the decomposition $\{ \mathcal{E}_{h_\gamma}(\gamma)\}_{h_{\gamma}}$ is a uniform decomposition of $\gamma$ that use $N_{\gamma}$ intervals and we define $h_{\gamma}$ as $\frac{1}{N_{\gamma}}$.

\subsection{Drift-Diffusion Poisson coupling}

In this paper, we consider that the surface particle density $\rho$ is solution of  the classical stationary drift-diffusion equation
\begin{equation}\label{eq_DD}
J'(x)=0, \hspace{1cm}\hbox{with}  \hspace{1cm}  J(x) = q \mu \Big(U_T\rho'(x) - \rho(x)  u'_{\gamma}(x) \Big),
\end{equation}
completed with the neutrality boundary conditions
\begin{equation}\label{BC_DD}
\rho(0)=\rho(L)=N^+_{dop}.
\end{equation}
In the expression of the electron current density $J$, $q$ is the elementary charge, $k_B$ the Boltzmann constant, $T$ the temperature taken equal to $77$ K, $U_T=\frac{k_BT}{q}$ the thermal potential and $\mu$ the (constant) electron mobility that we choose equal to $4.5\times 10^{3}$ cm$^2$.V$^{-1}$.s$^{-1}$ as proposed in \cite{DBP10}. 

Different transport models, that have been recently derived or investigated, can be used to perform accurate self-consistent simulations of a GFET. For instance, a bipolar drift-diffusion model with peculiar mobility functions deduced from semiclassical Boltzmann equations have been considered in \cite{NR21}. Quantum effects can be added to such a drift-diffusion model or to a hydrodynamical models (see e.g. \cite{ZB11,ZJ13,LR19}). At a microscopic level, a full quantum description can be done using a Dirac-like equation that describes the chiral character of massless fermions in graphene \cite{CGPNG,NBCD10}. A phase-space formulation can also be considered thanks to the Wigner formalism as proposed for instance in \cite{MS11}. Finally, we mention that different description levels can be spatially coupled deriving quantum interface conditions as done in \cite{BN18,BNNR21} in the case of graphene. In this work, since our aim is to focus on the numerical resolution of the Poisson equation and to present the efficiency of the proposed interface approach, we have chosen not to enrich the transport description and to perform (non realistic) self-consistent computations using \eqref{eq_DD}.

Notice that, at thermal equilibrium, the solution of \eqref{eq_DD}-\eqref{BC_DD} is explicitly expressed with respect to $u_{\gamma}$ by
\begin{equation}\label{expr_dens_eq}
\rho(x) = N^+_{dop} e^{\frac{u_{\gamma}(x)}{U_T}}.
\end{equation}
Out-of-equilibrium, we use the decomposition $\{ \mathcal{E}_{h_\gamma}(\gamma)\}_{h_{\gamma}}$ (same decomposition than the one used for $u^{\gamma}_{h_{\gamma}}$) to discretize equation \eqref{eq_DD} by means of a Scharfetter-Gummel scheme (see e.g. \cite{BMP87} for details).

To treat the nonlinearity of the Drift-Diffusion Poisson coupling, we use the linearized Gummel iterative process as in \cite{Gu64}. In this interface context, it amounts to solve, for a given $\rho^k$, the modified Poisson equation 
\begin{equation}\label{poisson_gummel}
\left \{
\begin{array}{r c l }
{\bf a}({\bf u}^{k+1},{\bf v}) - {\bf b}(\boldsymbol \lambda^{k+1},{\bf v}) + \frac{1}{U_T}\int_{\gamma} \rho^{k} u_{\gamma}^{k+1} v_{\gamma} \ dx  &=& \int_{\gamma} \rho^k v_{\gamma} \ dx + \frac{1}{U_T} \int_{\gamma} \rho^{k} u_{\gamma}^{k} v_{\gamma} \ dx\\
{\bf b}(\boldsymbol \mu^{k+1},{\bf u}^{k+1}) &=& 0
\end{array}
\right. ,
\end{equation}
iteratively followed by the resolution of the transport equation \eqref{eq_DD}. We emphasize that, in \eqref{poisson_gummel}, only the interface variable $u_{\gamma}$ appears in the additional terms (compared to \eqref{interf_VF}). Consequently,  the costly assembling of the matrix can be done once at the beginning of the code and, at each Gummel iteration $k$, only few entries have to be updated for taking care of the non-linearity,  as illustrated in Fig.\ref{Fig_spy}. The assembling cost is then comparable to the one for a linear problem, contributing to the computational efficiency of our interface approach.

\begin{figure}[h!]
\begin{center}
\includegraphics[scale=0.18, keepaspectratio=true]{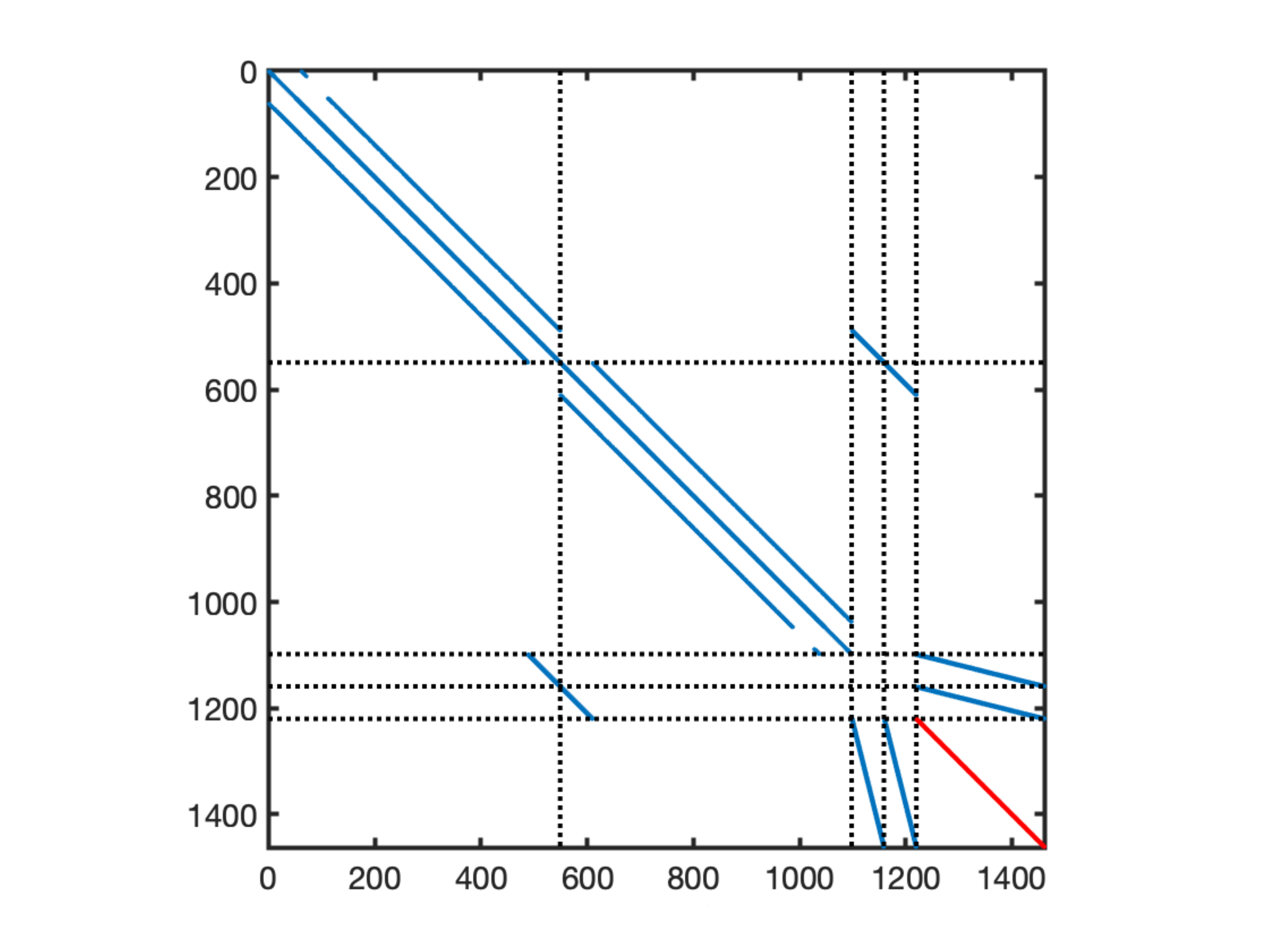}
\end{center}
\caption{Non zeros entries for the case $N_x=60$, $N_y=16$ and $N_{\gamma}=240$ (variable ordering: ($u^1_h,u^2_h,\lambda^1_h$, $\lambda^2_h$, $u^{\gamma}_{h_{\gamma}}$)). Entries affected by the Gummel iterative process are indicated in red.}
\label{Fig_spy}
\end{figure}

First, we present the potential and the density profiles obtained for our test case. Fig.\ref{Fig_pot_eq} represents the self-consistent potential at thermal equilibrium, showing that its shape is clearly driven by the chosen doping profile. Different applied voltages are then considered in Fig.\ref{Fig_pot_dens_vds} revealing the particle transport from Source to Drain.

\begin{figure}[h!]
\begin{center}
\includegraphics[scale=0.17, keepaspectratio=true]{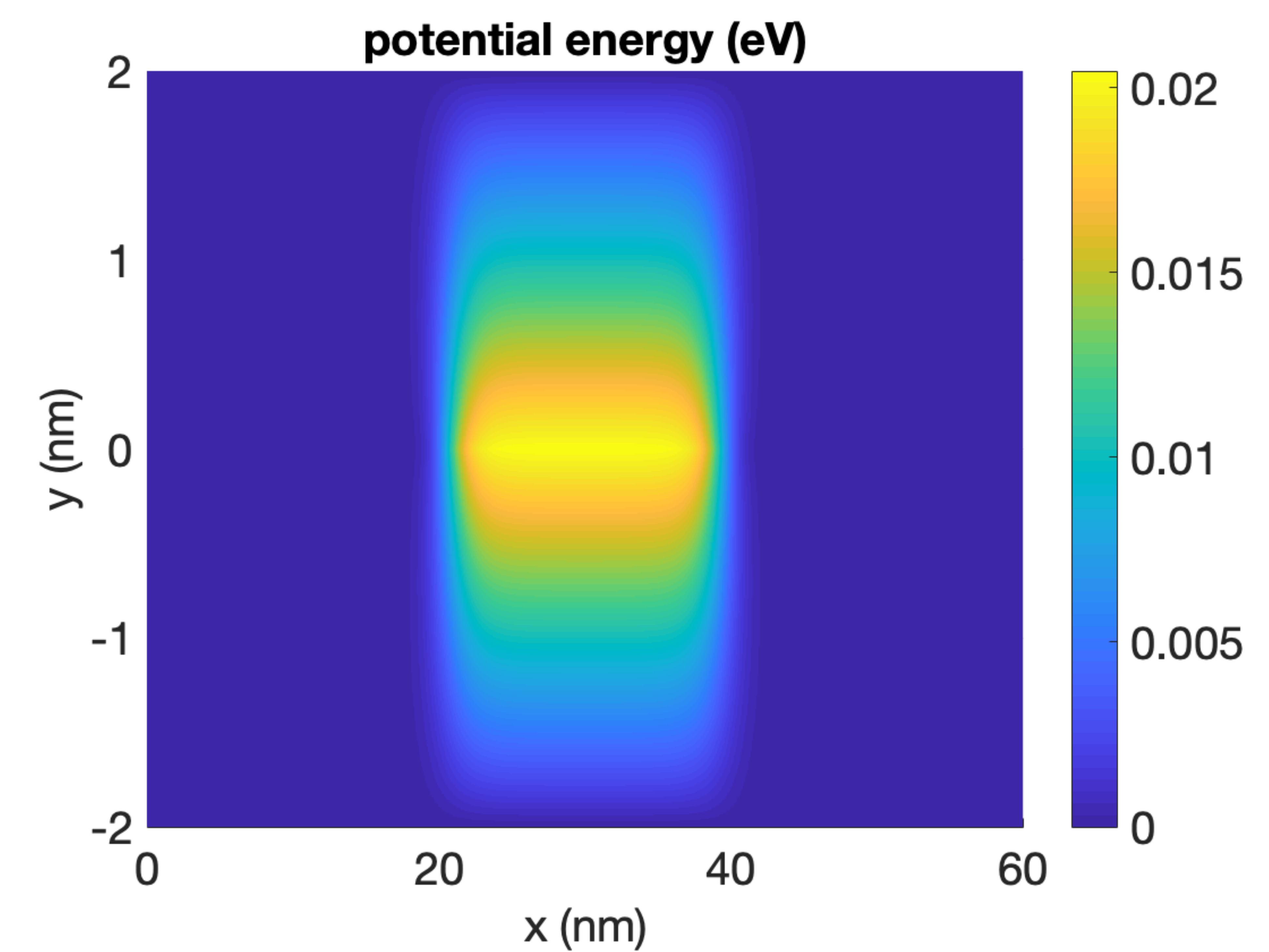}\hspace{0.6cm}
\includegraphics[scale=0.17, keepaspectratio=true]{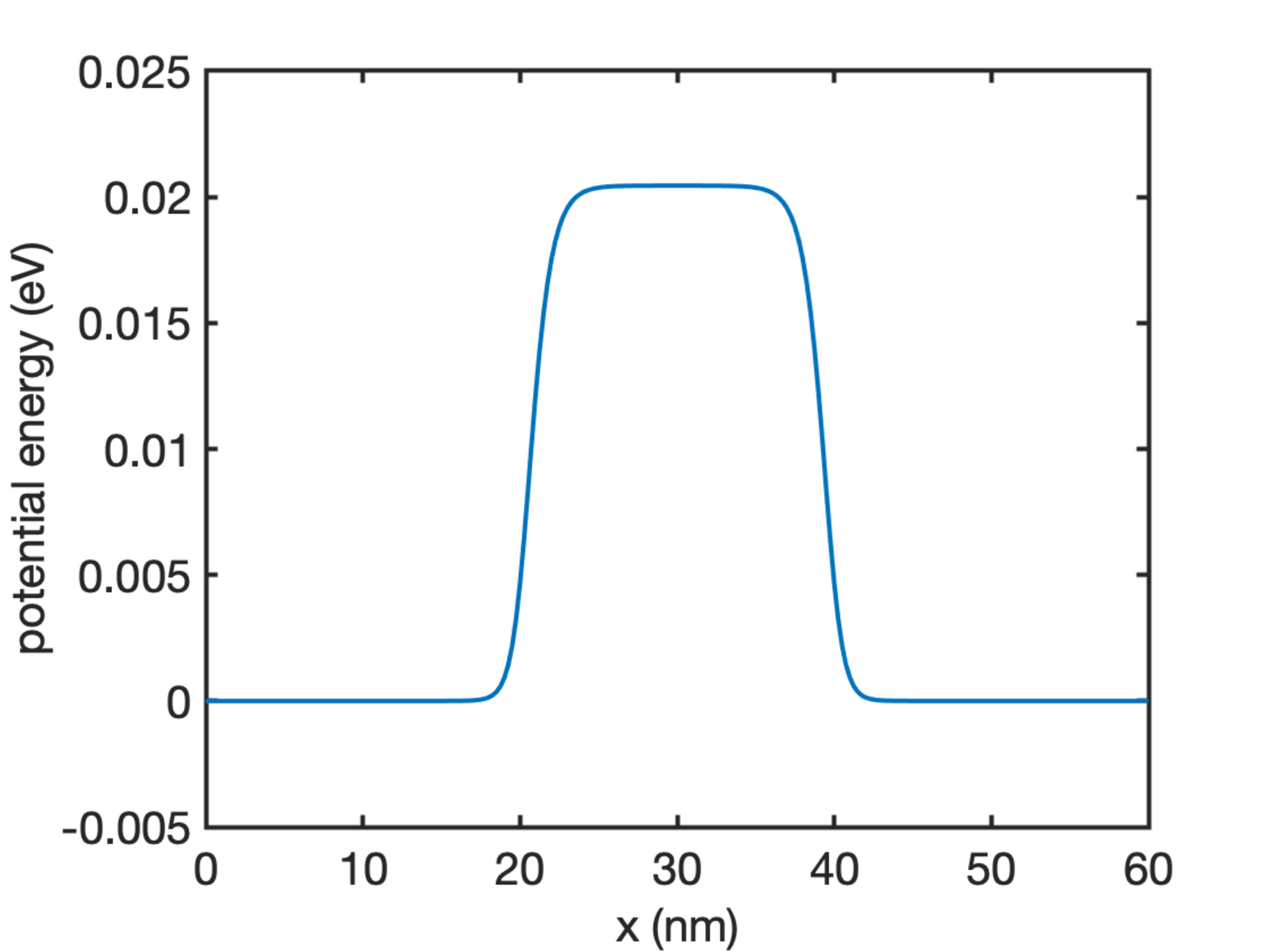}
\end{center}
\caption{2D potential energy $-u(x,y)$ (left) and interface potential energy $-u_{\gamma}(x)$ (right) at thermal equilibrium.}
\label{Fig_pot_eq}
\end{figure}

\begin{figure}[h!]
\begin{center}
\includegraphics[scale=0.17, keepaspectratio=true]{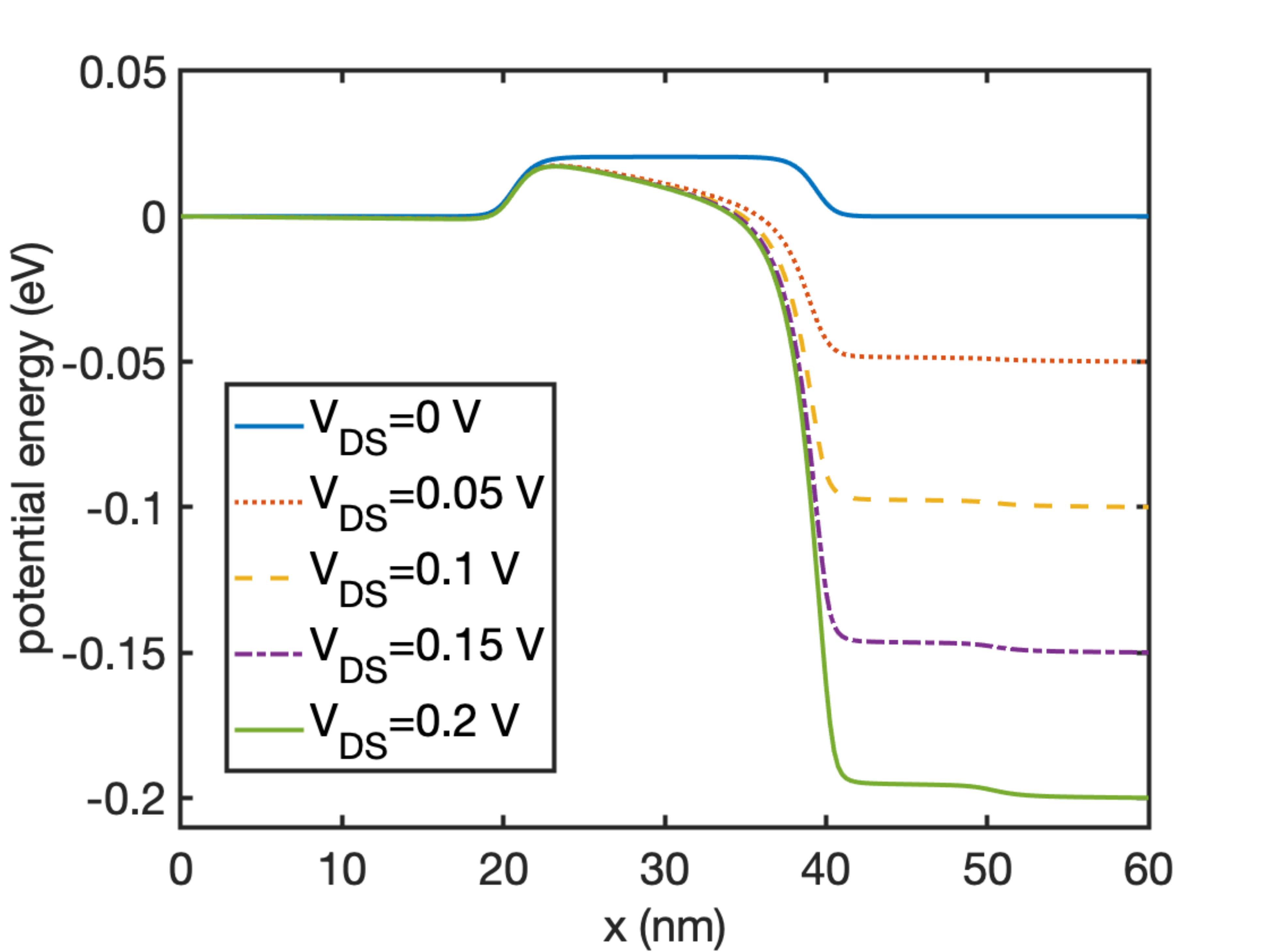}\hspace{0.6cm}
\includegraphics[scale=0.17, keepaspectratio=true]{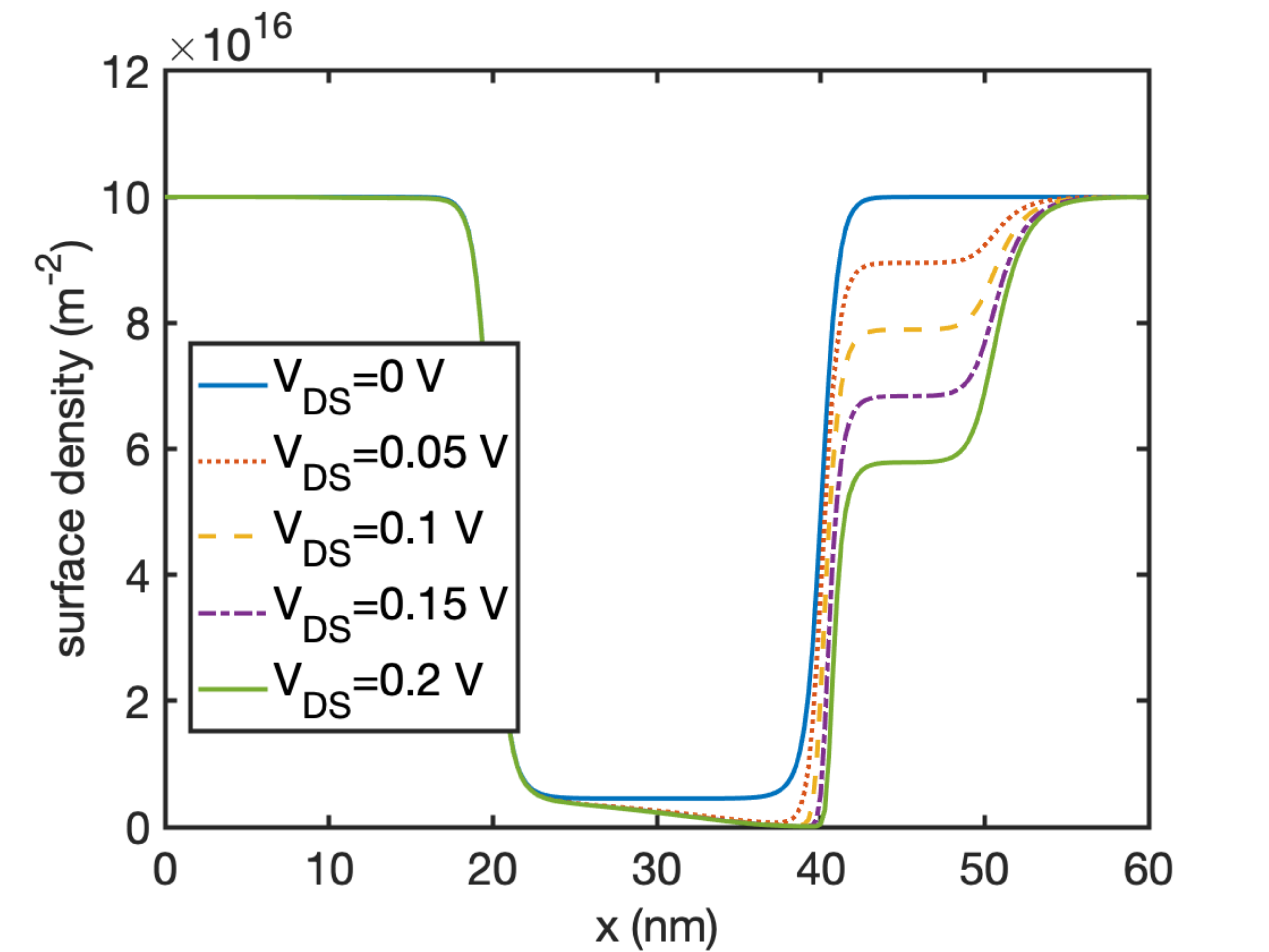}
\end{center}
\caption{Interface potential energy $-u_{\gamma}(x)$ (left) and surface density $\rho(x)$ (right) for different applied voltages.}
\label{Fig_pot_dens_vds} 
\end{figure}

\subsection{Convergence history}

{We study the numerical convergence of our approach. At thermal equilibrium, the density depends non linearly on the potential as expressed in \eqref{expr_dens_eq}. It fulfills  the properties of boundedness and monotonicity in the sense that for $u_{\gamma}, v_{\gamma} \in L^{\infty}(\gamma)$ there exist $\kappa_1$, $\kappa_2 >0$ such that
  \begin{equation} \label{nonlin:h1}
  \| \rho(u_{\gamma}) - \rho(v_{\gamma}) \|_{0,\gamma} \leq \kappa_1 \| u_{\gamma} - v_{\gamma} \|_{{0,\gamma}},
\end{equation}
\begin{equation} \label{nonlin:h2}
  \int_{\gamma} (\rho(u_{\gamma}) - \rho(v_{\gamma}) )(u_{\gamma}- v_{\gamma}) \,dx \geq \kappa_2 \| u_{\gamma} - v_{\gamma} \|_{{0,\gamma}}^2.
\end{equation}
Thus, results obtained in Theorem \ref{Theorem_discrete} (and therefore in Theorem \ref{Th_final_err_est}) can be easily extended to this case. In addition, we are also interested in studying numerically the behavior for non-zero applied voltages.\\

\begin{figure}[b!]
\begin{center}
\includegraphics[scale=0.2, keepaspectratio=true]{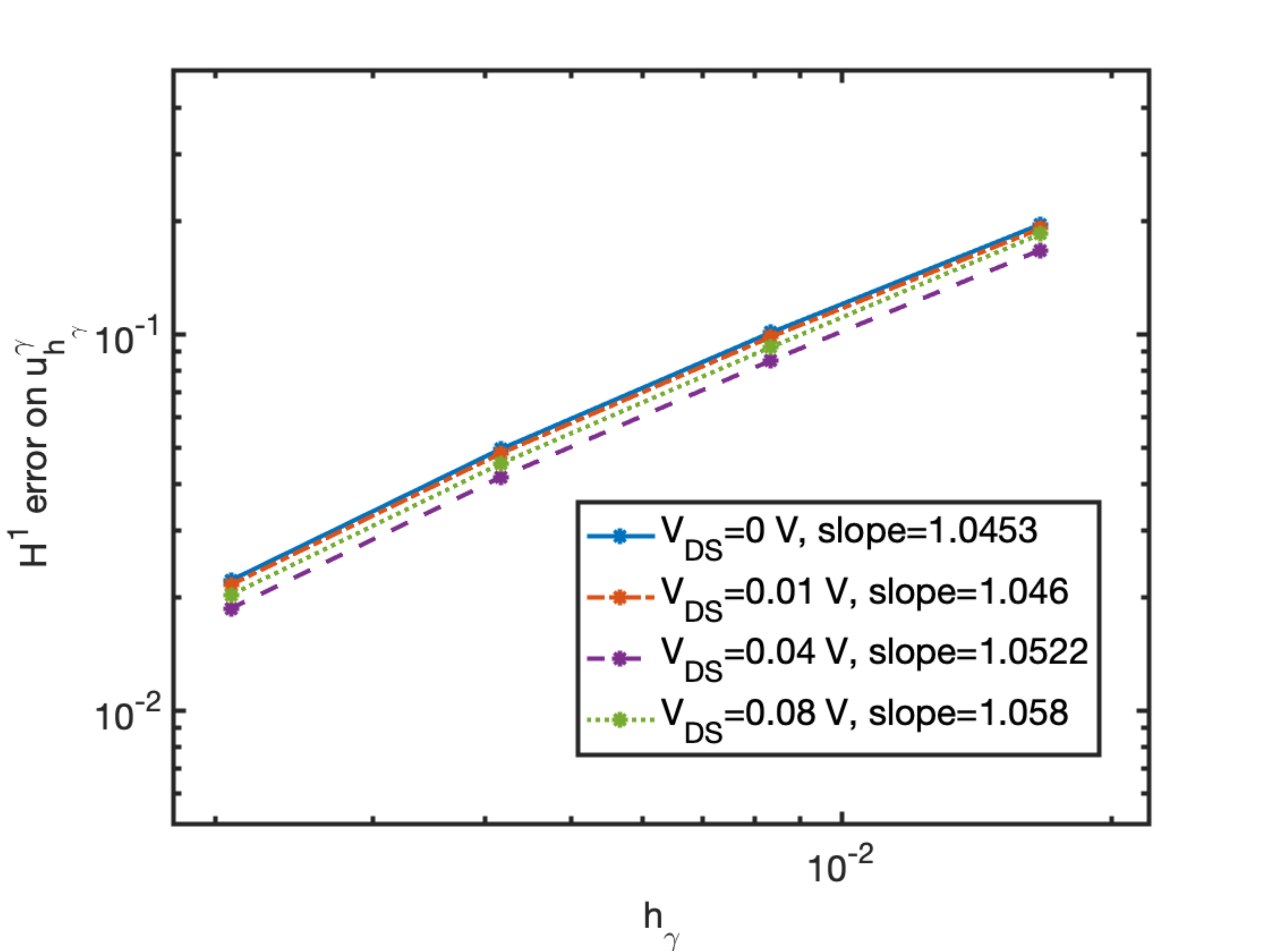}
\includegraphics[scale=0.2, keepaspectratio=true]{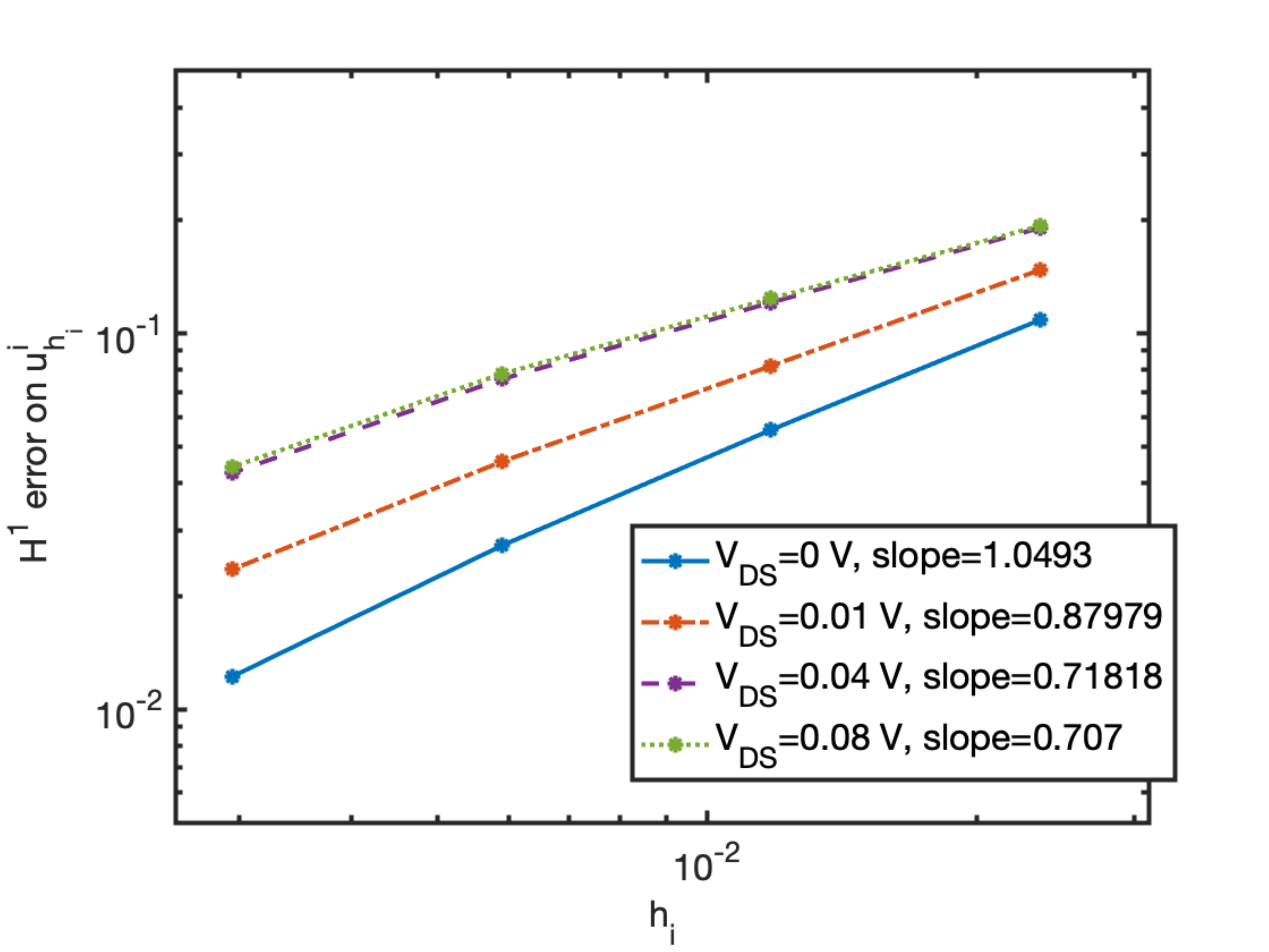}
\end{center}
\caption{Relative $H^1$ errors with respect to $h$ in logarithmic scale for the interface component (left) and the oxide ones (right).}
\label{Fig_err_est}
\end{figure}

We compute the solution for different meshes defined by $N_x=N_{\gamma}=60\times 2^i$ and $N_y=2^{i+2}$, $i=0,\dots,4$ and we choose the one obtained for $i=4$ ($N_x=N_{\gamma}=960$ and $N_y=64$) as reference solution. Relative errors corresponding to
\begin{equation}
\mathcal{E}_{1D} = \frac{\| u^{\gamma}_{h_{\gamma,ref}} - u^{\gamma}_{h_{\gamma}}\|_{1,\gamma} }{\| u^{\gamma}_{h_{\gamma,ref}}\|_{1,\gamma} } \hspace{1cm}\hbox{and} \hspace{1cm} \mathcal{E}_{2D} =\frac{\Big(\sum_{i=1}^2 \| u^i_{h_{i,ref}} - u^{i}_{h_i} \|_{1,\Omega_i}^2\Big)^{1/2}}{\Big(\sum_{i=1}^2 \| u^i_{h_{i,ref}} \|_{1,\Omega_i}^2\Big)^{1/2}},
\end{equation} 
are presented in Fig.\ref{Fig_err_est}, looking respectively at the interface component $u^{\gamma}_{h_{\gamma}}$ (left) and at the oxide components $u^i_{h_i}$ (right). As expected, straight lines of slope 1 are obtained for thermal equilibrium, in both cases. For non-zero applied voltages, we observe however a slightly lower slope for the oxide components. This behavior can be explained by a deterioration of the regularity of the solution at the junctions between the Neumann and the Dirichlet boundary conditions (and in particular at gate extremities). When considering the $L^{\infty}$ errors presented in Fig.\ref{Fig_err_est_Linf} the loss in the convergence rate for  non-zero applied voltages is more  evident. Indeed,  the maximum error is located along $y=0$ at the end of the channel for thermal equilibrium and at $y=\pm\frac{l}{2}$ around the gate extremity for non-zero applied voltages.

\begin{figure}[t!]
\begin{center}
\includegraphics[scale=0.2, keepaspectratio=true]{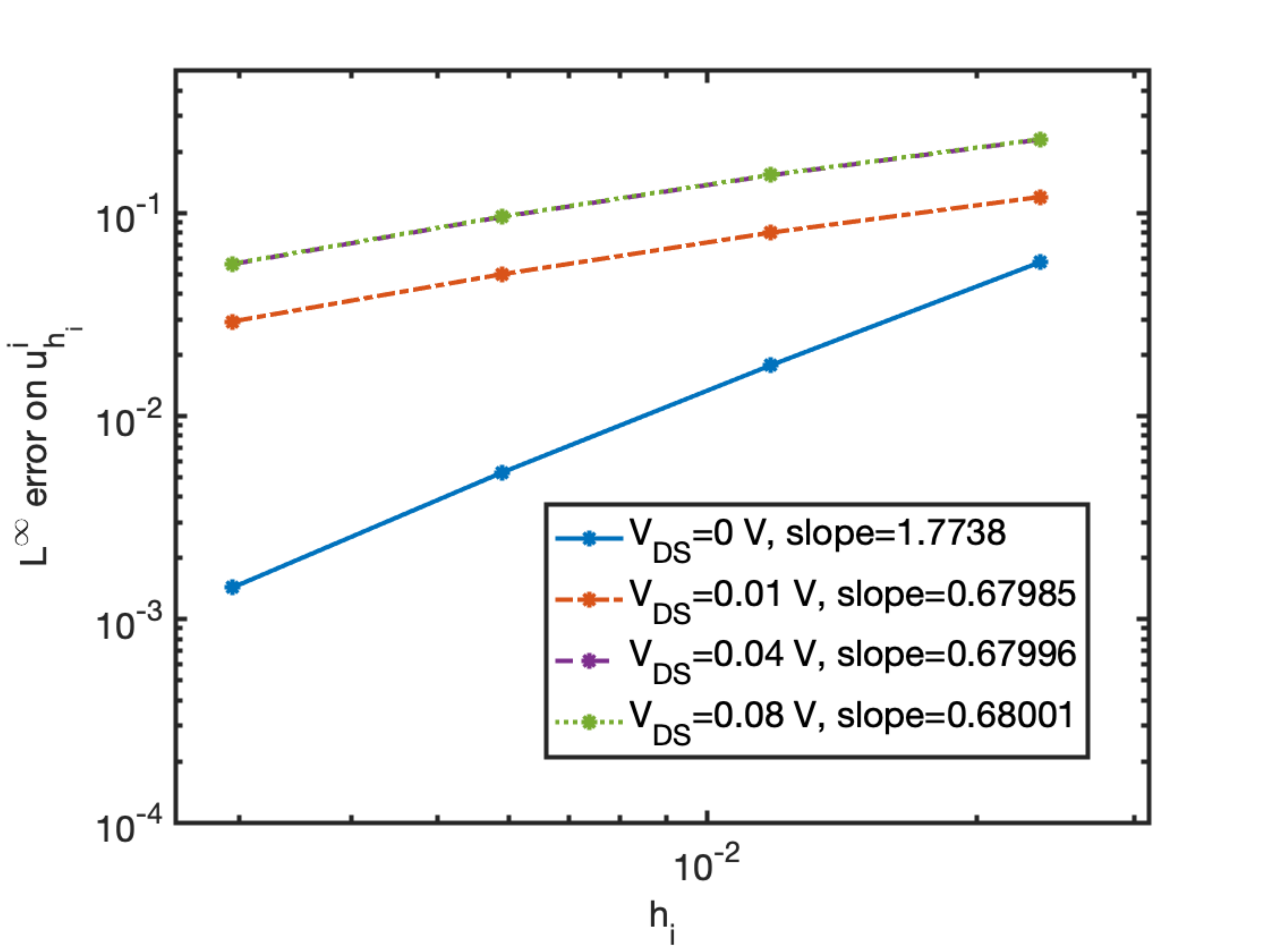}
\end{center}
\caption{Relative $L^{\infty}$ errors with respect to $h$ in logarithmic scale for the oxide components.}
\label{Fig_err_est_Linf}
\end{figure}

In self-consistent computations it is also interesting to look at the error behavior for the density, as done in Fig.\ref{Fig_err_est_dens}. The rate of convergence for the $H^1$ error is shown to be 1, both at thermal equilibrium (i.e. when $\rho$ is explicitly expressed with respect to $u_{\gamma}$ by  \eqref{expr_dens_eq})  and  with different applied voltages.\\

\begin{figure}[h!]
\begin{center}
\includegraphics[scale=0.2, keepaspectratio=true]{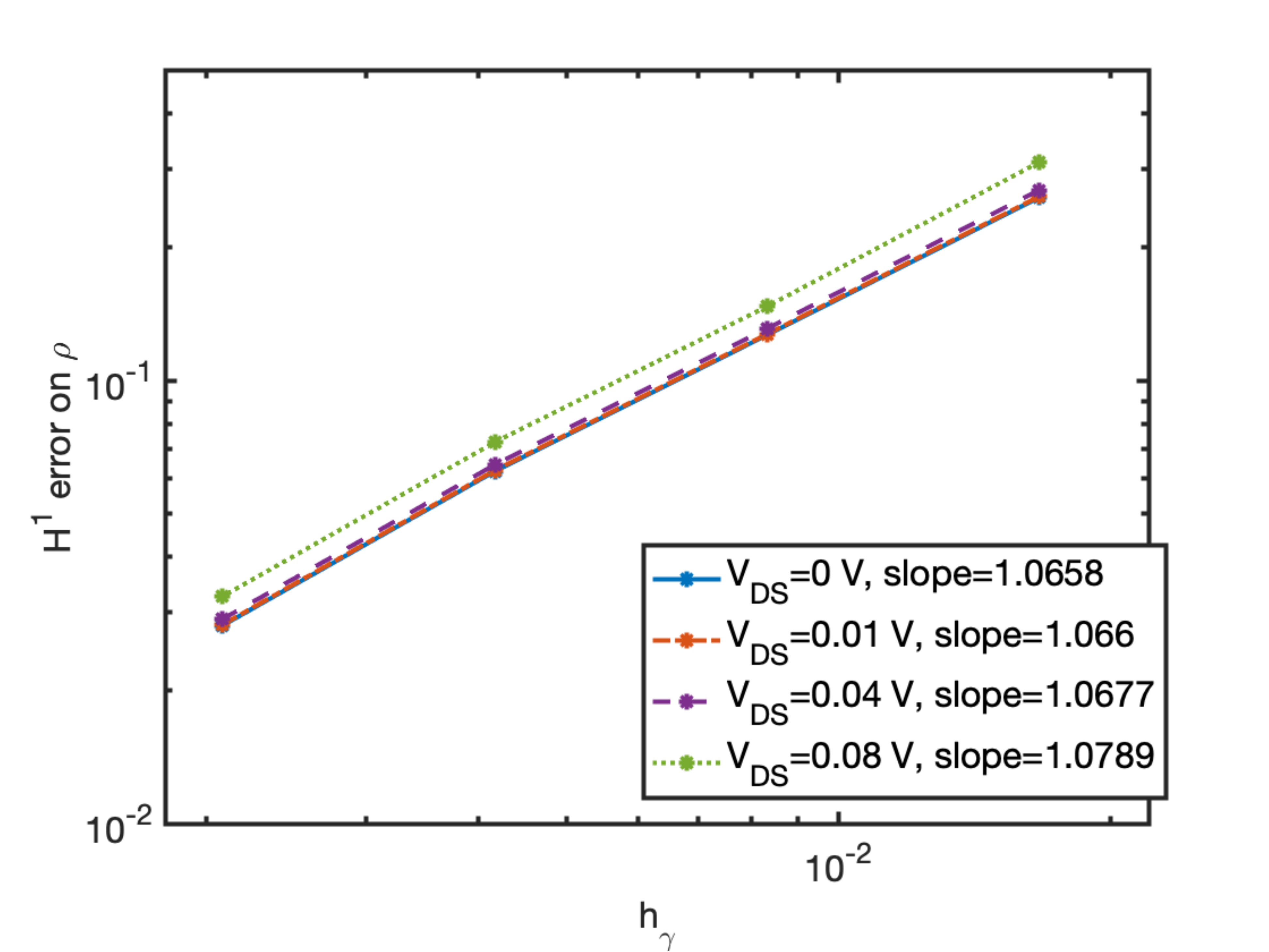}
\end{center}
\caption{Relative $H^1$ errors with respect to $h$ in logarithmic scale for the surface density $\rho$. }
\label{Fig_err_est_dens}
\end{figure}

\subsection{Number of discretization points}

Next, we discuss the effect of the number of discretization points $N_y$, $N_x$ and $N_{\gamma}$ on the error.  We present results at thermal equilibrium and for $V_{DS}=0.04$ V. Similar results are obtained for other non-zero applied voltages. The reference solution is chosen as in the previous subsection.

First, we fix $N_x=N_{\gamma}$ and we look at the errors varying $N_y$.  Results are presented in Table \ref{Table_Ny} for $N_x=N_{\gamma}=60$. The process bringing to the interface model reduces the importance of the transversal discretization around the single layer material and, indeed, we observe that the choice of $N_y$ slightly affects the result. It is especially true when we do not consider the coarser mesh ($N_y=8$). That is why, in the following, we fix $N_y=16$ and focus on the discretization along the transport direction. \\

\begin{table}[h!]
\begin{center}
\begin{tabular} {|c||c||c||c|c|}
\cline{2-5}
 \multicolumn{1}{c}{} 
& \multicolumn{2}{|c||}{$V_{DS}=0\ V$} & \multicolumn{2}{|c|}{$V_{DS}=0.04\ V$} \\
\hline
$N_y$ & $\mathcal{E}_{1D}$ & $\mathcal{E}_{2D}$ & $\mathcal{E}_{1D}$ & $\mathcal{E}_{2D}$  \\
\hline \hline
8 &  1.9557e-01   & 8.7731e-02 & 1.6695e-01 & 1.6448e-01 \\
\hline
16 & 1.9567e-01   & 8.1467e-02 & 1.6707e-01  & 1.5721e-01 \\ 
\hline
32 &  1.9570e-01  &  7.8688e-02 & 1.6711e-01 & 1.5496e-01 \\
\hline
64  & 1.9571e-01  & 7.7546e-02 & 1.6713e-01 & 1.5387e-01\\
\hline
\end{tabular}
\end{center}
\caption{Relative $H^1$ errors varying $N_y$ for $N_x=N_{\gamma}=60$.}
\label{Table_Ny}
\end{table}

As we mentioned, an interesting point of this approach is that the interface grid does not need to match with the one of the oxide subdomains. In other words, $N_{\gamma}$ can be chosen different from $N_x$. In Tables \ref{Table_Nx=Ng} and \ref{Table_Nx_neq_Ng}, we present the $H^1$ errors for the interface component $u^{\gamma}_{h_{\gamma}}$ as well as the density $\rho$ for a fixed $N_y=16$ either taking $N_x=N_{\gamma}$ (Table \ref{Table_Nx=Ng}) or fixing $N_x$ and varying $N_{\gamma}$ (Table \ref{Table_Nx_neq_Ng}).

\begin{table}[h!]
\begin{center}
\begin{tabular} {|c||c||c||c||c|}
\cline{2-5}
 \multicolumn{1}{c}{} 
& \multicolumn{2}{|c||}{$V_{DS}=0\ V$} & \multicolumn{2}{|c|}{$V_{DS}=0.04\ V$} \\
\hline
$N_x=N_{\gamma}$  &  $\mathcal{E}_{1D}$ & $\frac{\| \rho - \rho_{ref}\|_{1,\gamma} }{\| \rho_{ref}\|_{1,\gamma} }$ &  $\mathcal{E}_{1D}$ & $\frac{\| \rho - \rho_{ref}\|_{1,\gamma} }{\| \rho_{ref}\|_{1,\gamma} }$ \\
\hline \hline
60 &   1.9567e-01    & 2.5930e-01 & 1.6707e-01 &  2.6920e-01 \\ \hline
120 &  1.0079e-01  & 1.2674e-01 &  8.4934e-02 & 1.3105e-01  \\ \hline
240 &   4.9450e-02  &  6.2400e-02 & 4.1631e-02  & 6.4557e-02 \\  \hline
480 &  2.2167e-02  &   2.7980e-02 & 1.8681e-02 & 2.8969e-02 \\ \hline
\end{tabular}
\end{center}
\caption{Relative $H^1$ errors of the interface potential $u^{\gamma}_{h_{\gamma}}$ and the density $\rho$ in the case $N_x=N_{\gamma}$.}
\label{Table_Nx=Ng}
\end{table}

\begin{table}[h!]
\begin{center}
\begin{tabular} {|c|c||c||c||c||c|}
\cline{3-6}
 \multicolumn{2}{c}{}  & \multicolumn{2}{|c||}{$V_{DS}=0\ V$} & \multicolumn{2}{|c|}{$V_{DS}=0.04\ V$} \\\hline
$N_x$ & $N_{\gamma}$ & $\mathcal{E}_{1D}$  & $\frac{\| \rho - \rho_{ref}\|_{1,\gamma} }{\| \rho_{ref}\|_{1,\gamma} }$ & $\mathcal{E}_{1D}$  & $\frac{\| \rho - \rho_{ref}\|_{1,\gamma} }{\| \rho_{ref}\|_{1,\gamma} }$ \\
\hline \hline
& 30 &  3.3376e-01   &  5.4465e-01  & 2.7736e-01 & 5.5921e-01  \\ \cline{2-6}
& 120 &  1.0177e-01   &  1.2709e-01 & 8.7112e-02  & 1.3215e-01   \\ \cline{2-6}
60 &  240 &  5.3930e-02  & 6.4590e-02 & 4.8621e-02 &  6.9023e-02  \\ \cline{2-6}
& 480 &  3.2494e-02   & 3.3796e-02  & 3.2738e-02 & 3.9536e-02    \\ \cline{2-6}
& 960 &  2.4381e-02   & 1.9597e-02 & 2.7391e-02 & 2.7634e-02 \\
\hline
\end{tabular}

\end{center}
\caption{Relative $H^1$ errors of the interface potential $u^{\gamma}_{h_{\gamma}}$ and the density $\rho$  for a given $N_x$ and different $N_{\gamma}$.}
\label{Table_Nx_neq_Ng}
\end{table}

We observe  that the error for the 1D component decreases when increasing the grid points on the interface. In particular, the errors obtained for $N_x=60$ and $N_{\gamma}=240$ are smaller than the ones for $N_x=N_{\gamma}=60$ and comparable to the ones for  $N_x=N_{\gamma}=240$. Since the cost of solving the linear system is driven by the number of degrees of freedom in the oxide region, it is therefore very appealing to use a relatively coarse mesh in the oxide region and a finer grid on the interface. 

\subsection{Anisotropic permittivity}\label{section_num_anisotropy}

We now consider the case where the channel dielectric permittivity is given by the diagonal tensor \eqref{perm_tensor}. We have seen that the effective equation \eqref{eq_interf} contains only the in-plane permittivity $\epsilon_{//}$ and that a possibility to retain the information of the out-of-plane permittivity  $\epsilon_{\perp}$ is to replace the Dirichlet type continuity conditions \eqref{eq_continuity} by the Robin type condition \eqref{eq_continuity2}. In this part, we perform a comparison between these two continuity conditions. To estimate the differences, we 
compare them with an approximate solution of the transmission problem \eqref{eq1_strip}-\eqref{cond_transm_strip} obtained with a very fine mesh. For that, the delta function in the second member of equation \eqref{eq2_strip} is approximated by $\frac{1}{a\sqrt{\pi}} e^{-(y/a)^2}$ with $a=0.008$ nm (25 times smaller than $d$). Since the numerical convergence and the effect of the number of discretization points discussed in the previous subsections are not affected by the choice of the continuity conditions, we do not present them again. Instead, we concentrate on the vertical potential slice $u(\frac{L}{2},y)$ at thermal equilibrium and on the current-voltage characteristics to analyze the effect of the continuity conditions.

Obtained results are presented in Figs. \ref{Fig_pop_epst_normal} and \ref{Fig_IV_epst_normal} for the case $N_x=N_{\gamma}=240$ and $N_y=16$. Solid blue lines correspond to the transmission problem, dashed red lines to the interface approach with the continuity condition \eqref{eq_continuity} and dashdotted purple lines to the interface approach with \eqref{eq_continuity2}. We observe only slightly changes between the three approaches, both at thermal equilibrium in Fig.\ref{Fig_pop_epst_normal} and with applied voltages in Fig.\ref{Fig_IV_epst_normal}. In particular, for $V_{DS}=0$ V, we have $|u_{ch}(\frac{1}{2},0) - u_{\gamma}(0)|=3.0\times 10^{-5}$ with \eqref{eq_continuity} and $|u_{ch}(\frac{1}{2},0) - u_{\gamma}(0)|=9.83\times 10^{-5}$ with \eqref{eq_continuity2}. 

\begin{figure}[t!]
\begin{center}
\includegraphics[scale=0.2, keepaspectratio=true]{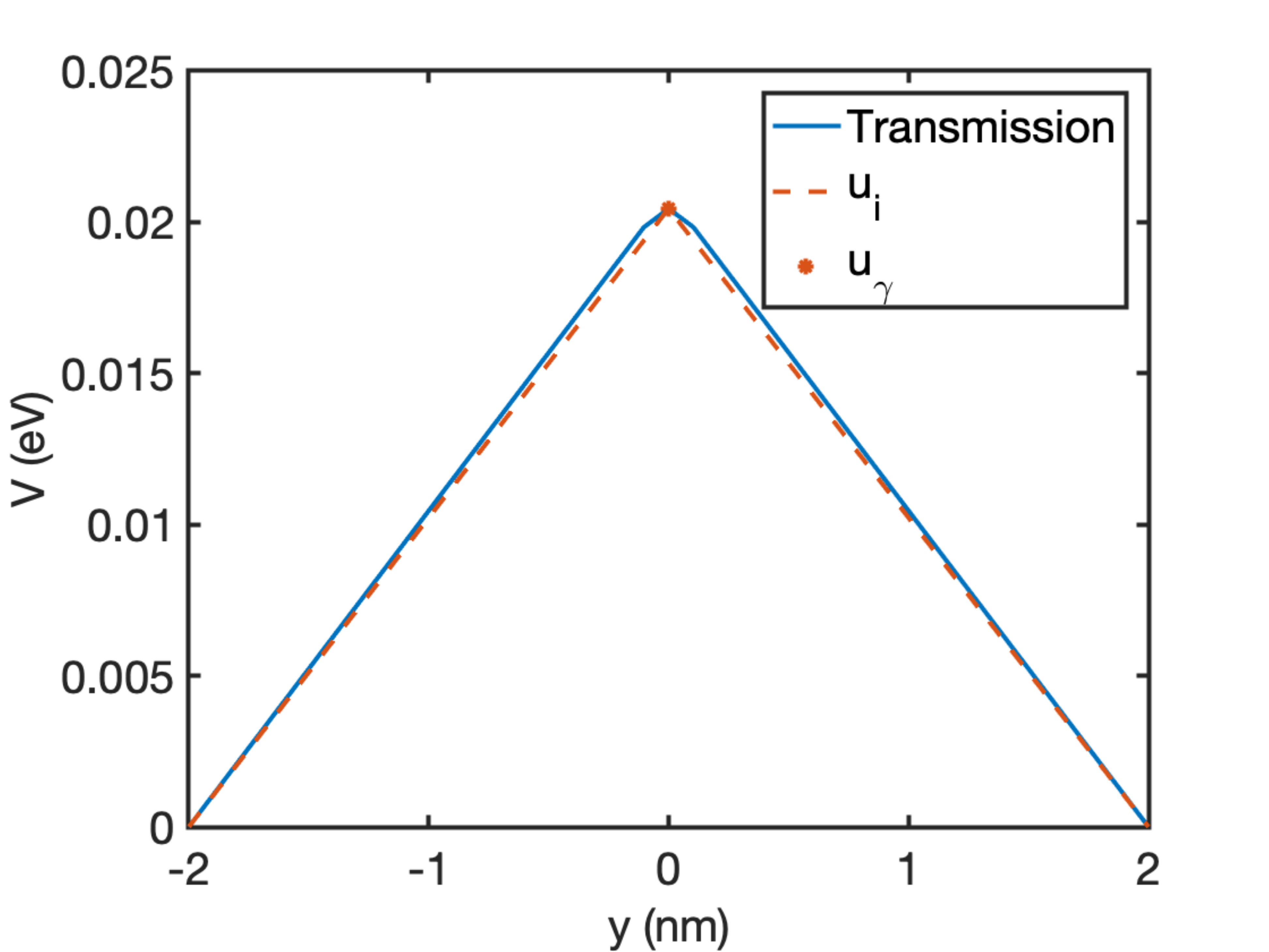}
\includegraphics[scale=0.2, keepaspectratio=true]{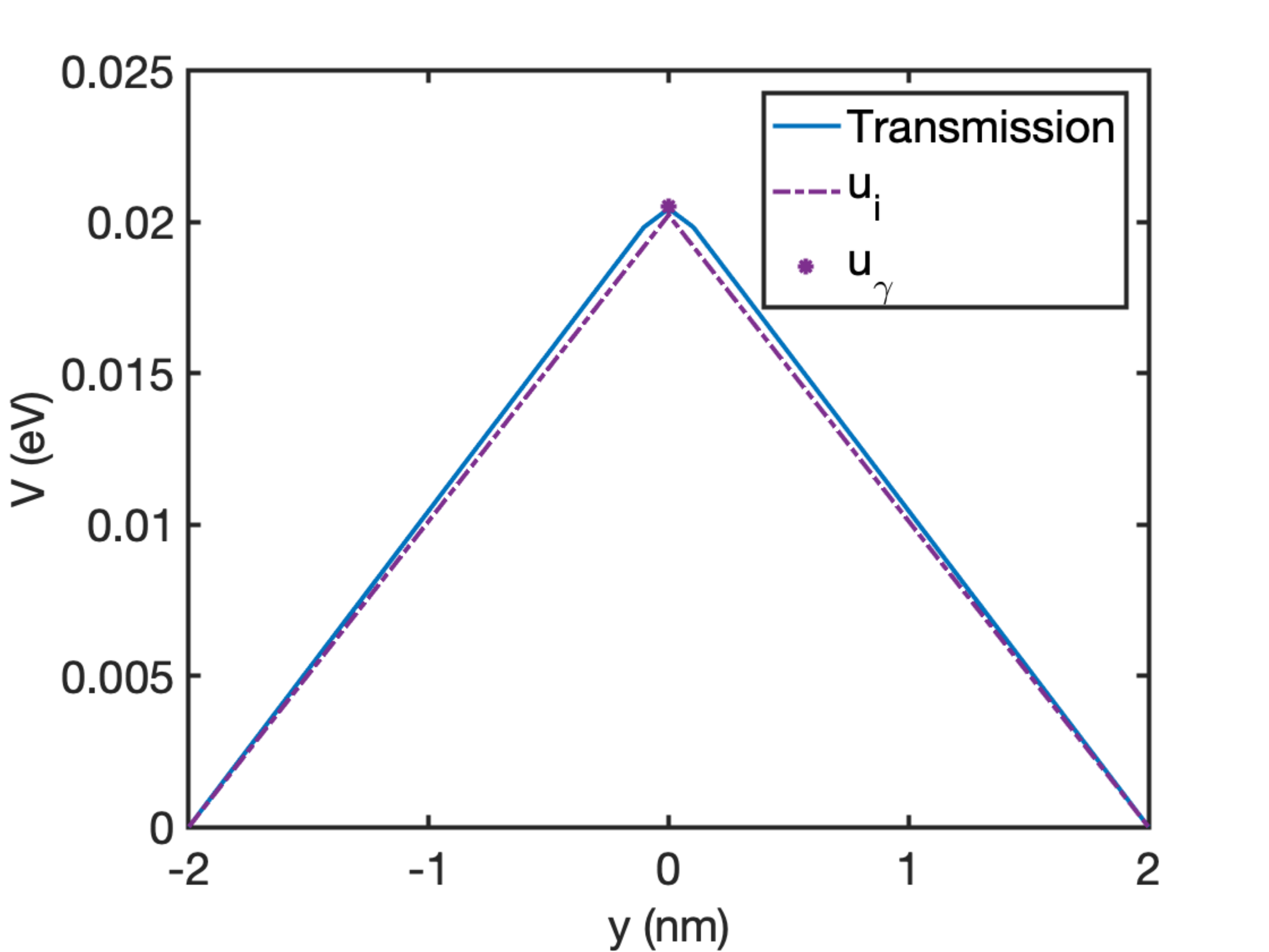}
\end{center}
\caption{Vertical potential slice $u(\frac{L}{2},y)$ at thermal equilibrium obtained for the transmission problem \eqref{eq1_strip}-\eqref{cond_transm_strip}    (solid blue line) and for the interface approach with condition \eqref{eq_continuity} (left) or condition \eqref{eq_continuity2} (right) for the case $\epsilon_{//}=13.9 \ \epsilon_0$ and $\epsilon_{\perp}=6.9 \ \epsilon_0$.}
\label{Fig_pop_epst_normal}
\end{figure}

\begin{figure}[t!]
\begin{center}
\includegraphics[scale=0.22, keepaspectratio=true]{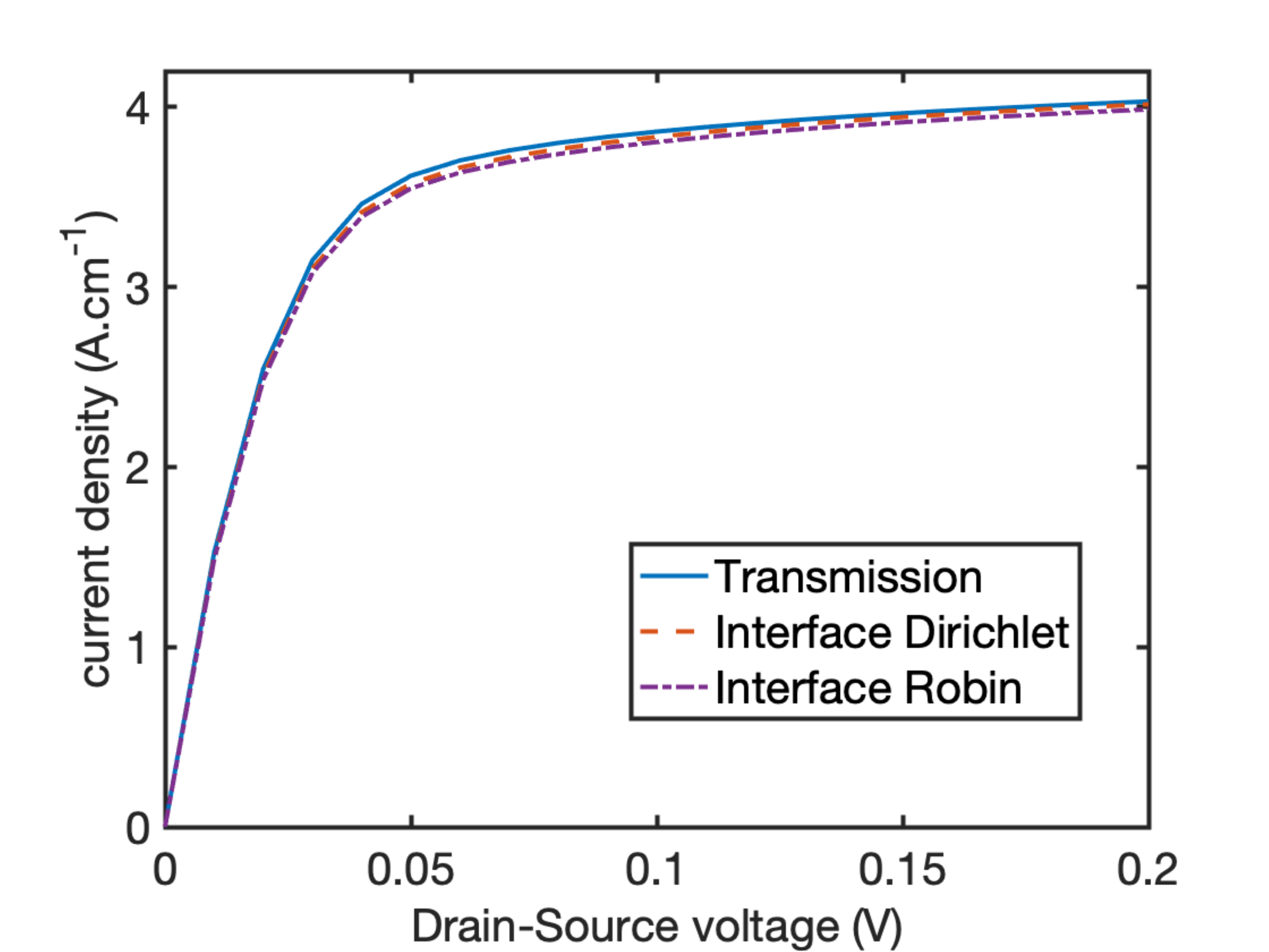}
\end{center}
\caption{Current-voltage characteristics for the case $\epsilon_{//}=13.9 \ \epsilon_0$ and $\epsilon_{\perp}=6.9 \ \epsilon_0$.}
\label{Fig_IV_epst_normal}
\end{figure}

However, to overemphasize the effect of the continuity condition, we then choose an artificial extreme out-of-plane permittivity $\epsilon_{\perp}=0.1 \ \epsilon_0$. Obtained results are presented in Figs. \ref{Fig_pop_epst01} and \ref{Fig_IV_epst01}. They clearly show that the discontinuity allowed by the condition \eqref{eq_continuity2} between the interface component $u_{\gamma}$ and the oxide components $u_{i}$ is essential to capture the effects due to a strong anisotropy.

\begin{figure}[h!]
\begin{center}
\includegraphics[scale=0.2, keepaspectratio=true]{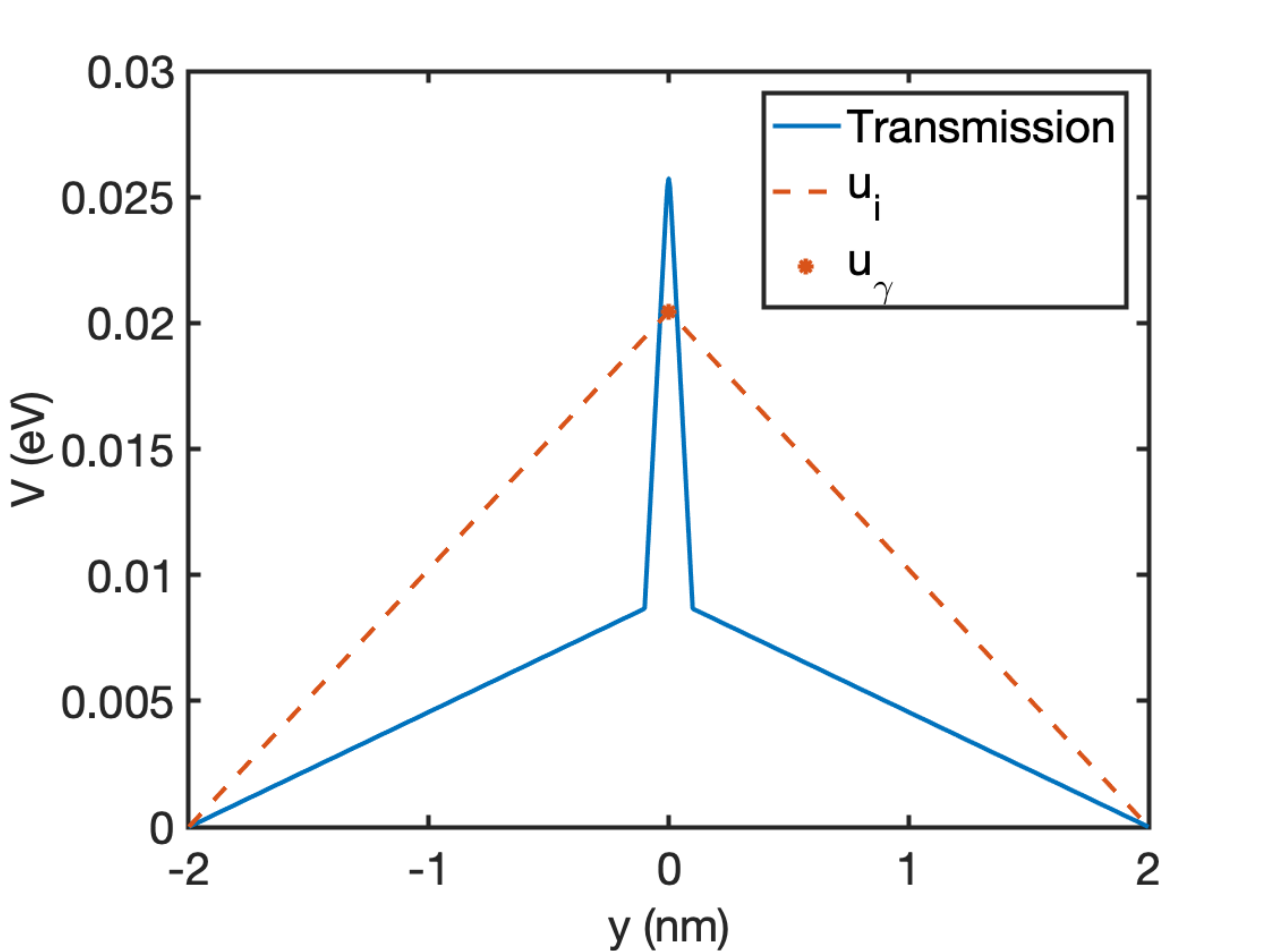}
\includegraphics[scale=0.2, keepaspectratio=true]{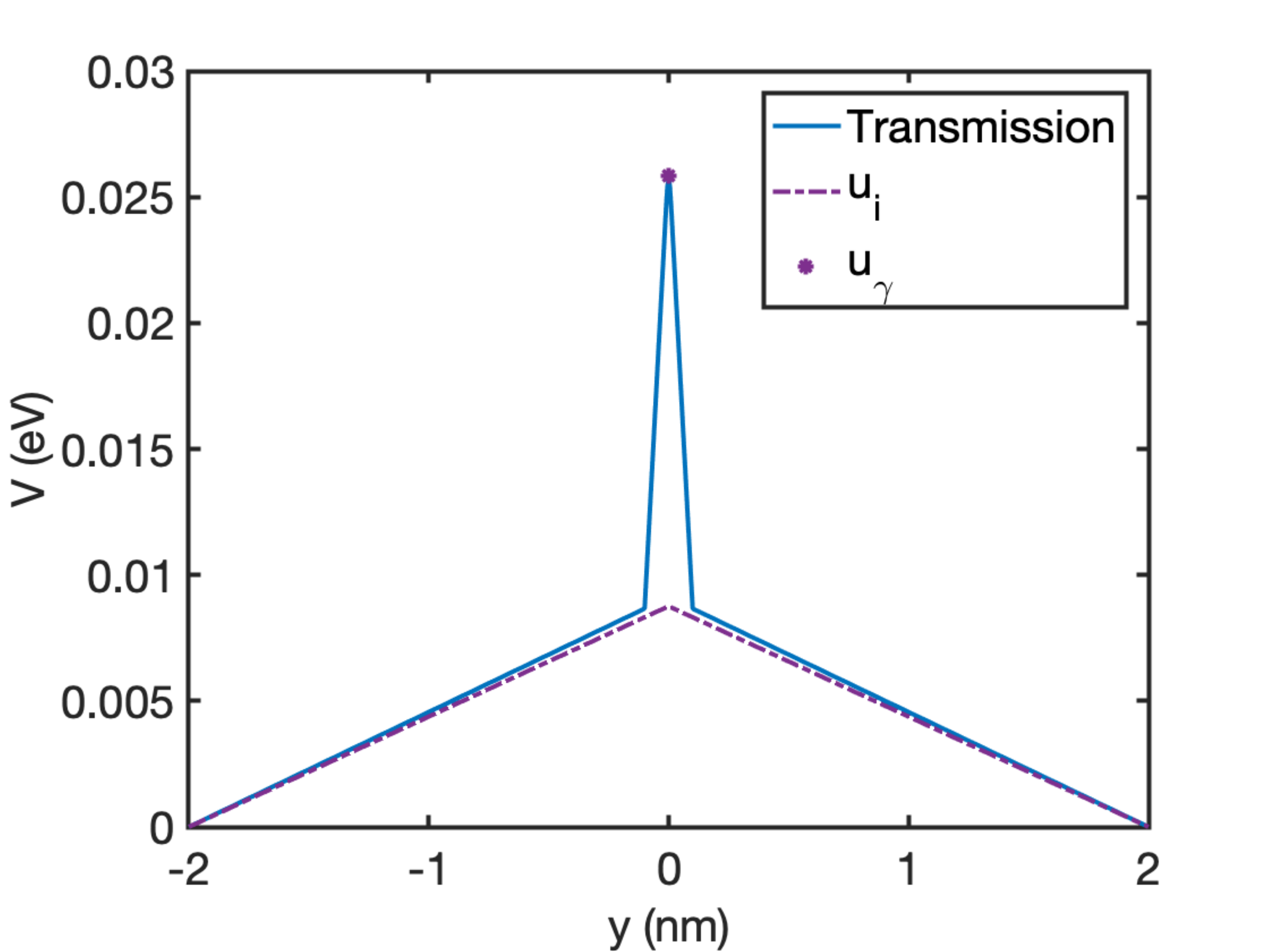}
\end{center}
\caption{Vertical potential slice $u(\frac{L}{2},y)$ at thermal equilibrium obtained for the transmission problem \eqref{eq1_strip}-\eqref{cond_transm_strip}    (solid blue line) and for the interface approach with condition \eqref{eq_continuity} (left) or condition \eqref{eq_continuity2} (right) for the case $\epsilon_{//}=13.9 \ \epsilon_0$ and $\epsilon_{\perp}=0.1 \ \epsilon_0$.}
\label{Fig_pop_epst01}
\end{figure}

\begin{figure}[h!]
\begin{center}
\includegraphics[scale=0.22, keepaspectratio=true]{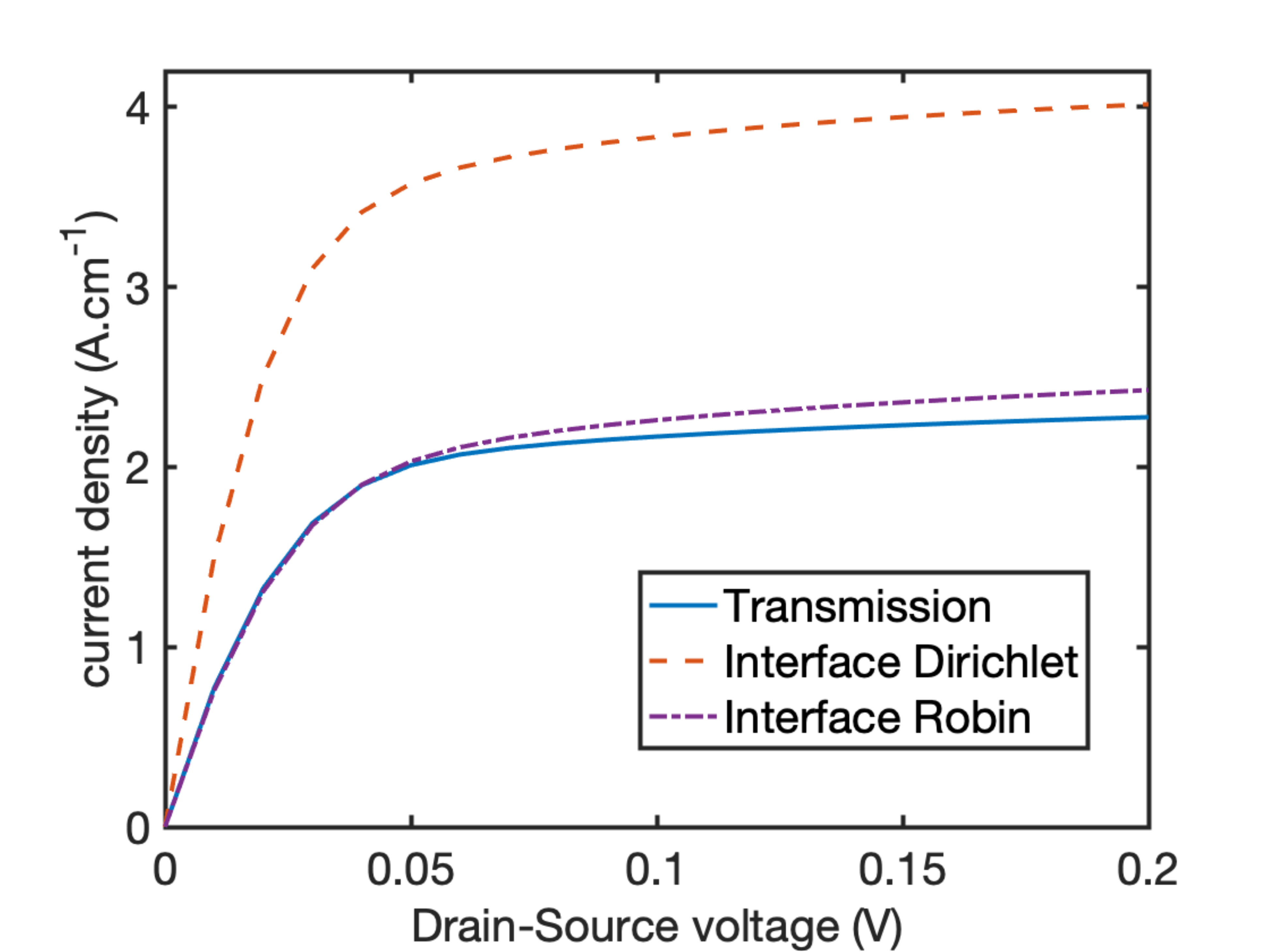}
\end{center}
\caption{Current-voltage characteristics for the case $\epsilon_{//}=13.9 \ \epsilon_0$ and $\epsilon_{\perp}=0.1 \ \epsilon_0$.}
\label{Fig_IV_epst01}
\end{figure}


\section{Conclusion}

We discussed on the numerical resolution of a Poisson equation describing the electrostatics  of  devices in the presence of a semiconducting single-layer material. The proposed interface approach provides a good framework for the mathematical analysis and for the approximation of its variational formulation. A Robin type continuity condition along $\gamma$ (with a Robin coefficient depending on the effective dielectric thickness $d$) can be imposed to consider out-of-plane/in-plane permittivities for a better description of the single-layer/oxide interactions.
The presented  numerical scheme has the advantage to avoid the need of a fine mesh in a 2D region around the single layer material. Moreover, the assembling of the associated matrix  is done at a  cost comparable with the linear case, even when a coupled transport-Poisson model is considered. Finally, it is worth mentioning the possibility of using a relatively coarse mesh in the oxide region and a finer grid on the interface. As continuation of this work, we expect to take great advantage of these interesting features of the interface approach for the resolution of the Poisson equation in the context of  a Dirac-Poisson coupling to perform self-consistent computations of a GFET with an enriched description of the particle transport.\\

\noindent {\bf Acknowledgments:} The first author acknowledges partial support of the IDEX-IRS project NUM-GRAPH ``NUMerical simulation of the electron transport through GRA\-PHene nanostructures'' funded by Univ. Grenoble Alpes and Grenoble INP. The second author acknowledges the financial support of Italian Ministry of University and Research (MUR) through the PRIN grant n. 201744KLJL.


{\footnotesize \bibliographystyle{plain}
\bibliography{MaBiblio}}

\end{document}